\documentclass[journal,twoside,print]{ieeecolor}
\usepackage{generic}
\usepackage{cite}
\usepackage{amsmath,amssymb,amsfonts}
\usepackage{graphicx}
\usepackage{booktabs}
\usepackage{adjustbox}
\usepackage{hyperref}
\hypersetup{hidelinks}
\usepackage{textcomp}
\makeatletter
\def\ps@titlepagestyle{%
  \def\@oddhead{}\def\@evenhead{}%
  \def\@oddfoot{}\def\@evenfoot{}}
\makeatother
\newtheorem{thm}{Theorem}
\newtheorem{lem}[thm]{Lemma}
\newtheorem{assum}[thm]{Assumption}
\newtheorem{defn}[thm]{Definition}
\newtheorem{rem}[thm]{Remark}
\newtheorem{proposition}[thm]{Proposition}
\newenvironment{pf}{\begin{proof}}{\end{proof}}

\begin{document}

\title{A Riccati Approach to Mixed $H_2/H_\infty$ Closed-Loop
Games for Infinite-Dimensional Stochastic Systems}

\author{Mingyang Shen, Weihai Zhang, Qingxin Meng, Maoning Tang
\thanks{Mingyang Shen, Qingxin Meng, and Maoning Tang are with the Department of Mathematical Sciences, Huzhou Normal University, Huzhou 313000, Zhejiang, China (e-mail: smy03012025@163.com; mqx@zjhu.edu.cn; tmorning@zjhu.edu.cn).}%
\thanks{Weihai Zhang is with the College of Automation, Shandong University of Science and Technology, Qingdao, China (e-mail: w\_hzhang@163.com).}%
\thanks{Qingxin Meng was supported by the National Natural Science Foundation of China (No.~12271158).}%
\thanks{Corresponding author: Maoning Tang.}}

\maketitle

\begin{abstract}
This paper studies a finite-horizon mixed $H_2/H_\infty$ feedback Nash
game for stochastic evolution equations on a separable Hilbert space.
The drift generator is unbounded, the remaining coefficients are bounded,
and the one-dimensional Brownian diffusion depends on the state, control,
and disturbance. The $H_2$ channel is an LQ state--control energy.
For the disturbance channel, the stochastic LQ uniform-convexity
characterization yields equivalence between strict induced $L^2$ attenuation
and unique strongly regular mild Riccati solvability. Simultaneous
bounded-generator approximation of the Lyapunov equation and the state
justifies quadratic identities for strongly continuous mild operator
solutions. These identities verify both Nash inequalities and full-output
strict attenuation from a strongly regular coupled Riccati pair.
An invertible feedback block and a contraction argument establish locally
unique coupled solutions on a sufficiently short terminal interval for
every positive attenuation level. For a specified stochastic heat-equation
model at $\gamma=0.09$, a uniform invariant rectangle further proves
full-horizon existence, bounded infinite-dimensional feedbacks, and
operator-norm spectral convergence. Numerical computations reproduce the
projected gains and compare selected best responses. Global coupled
solvability for general coefficients remains an explicit hypothesis.
\end{abstract}

\begin{IEEEkeywords}
Infinite-dimensional systems, Stochastic $H_2/H_\infty$ control, Closed-loop strategies, Nash differential game, Operator coupled differential Riccati equations, Stochastic bounded real lemma
\end{IEEEkeywords}

\section{Introduction}

Mixed $H_2/H_\infty$ control \cite{zhou1996robust} is commonly used to describe the balance between nominal quadratic performance and robust disturbance attenuation. In the present finite-horizon stochastic setting, however, the $H_2$ channel is understood as an $H_2$-type stochastic LQ criterion that penalizes the regulated-output energy and the control energy, rather than as the Hilbert--Schmidt norm of an input--output operator. The $H_\infty$ requirement imposes an upper bound on the induced stochastic $L^2$-gain from the disturbance to the regulated output.

Existing approaches include convex and
linear-matrix-inequality  \cite{haddad2022mixed,rotea1991mixed,lin2006h2,scherer1997multiobjective},
Stackelberg \cite{jungers2008stackelberg}, and
Nash-game  \cite{limebeer1994nash,wang2023stochastic,zhang2017stochastic_nash}. Limebeer,
Anderson, and Hendel \cite{limebeer1994nash} developed a nonzero-sum Nash-game
approach for deterministic mixed control, and Chen and Zhang
\cite{chen2004stochastic} treated stochastic systems with state-dependent noise.
Related extensions include discrete-time infinite Markov jump systems
\cite{liu2017infinite} and continuous-time mean-field systems with
$(x,u,v)$-dependent noise \cite{wang2023stochastic}. These works are
finite-dimensional in their state description or concern a different time/noise
structure from the stochastic evolution equation studied here.

Open-loop controls may depend on the initial pair, whereas a closed-loop
strategy is prescribed independently of the initial state
\cite{yong1999stochastic,sunyong2020book}. Their solvability conditions differ
even for stochastic LQ control \cite{sun2016open}. Sun and Yong
\cite{sun2019linear} studied the corresponding distinction for finite-dimensional
nonzero-sum Nash games \cite{basar1999dynamic}.

For infinite-dimensional stochastic LQ control, L\"{u} \cite{lu2019well}
established equivalence between uniform convexity and unique strongly regular
mild Riccati solvability. Ma, Xiao, and Meng \cite{ma2025open} studied open-loop
and closed-loop saddle points for zero-sum evolution games with Poisson jumps.
Their mild-solution and Yosida-approximation methods guide the treatment of
unbounded generators here, but their game and noise structures differ from
the present Brownian-driven nonzero-sum problem. Related Hilbert-space
mean-field games are studied in \cite{liu2025hilbert}.

Distributed parameter systems naturally lead to Hilbert-space models
\cite{lasiecka2000control}. The operator Riccati and stochastic evolution
frameworks are classical \cite{fattorini1999infinite,da_prato2014stochastic,bensoussan1992stochastic},
as is the Yosida-regularized It\^o calculus \cite{govindan2016yosida}.
For $(x,u,v)$-dependent noise, Zhang, Zhang, and Chen \cite{zhang2006xuv}
derived cross-coupled Hamilton--Jacobi equations and, in the linear case,
matrix Riccati equations. Extending this formulation to an unbounded generator
requires mild operator equations, strong regularity, and justified quadratic
identities. Single-channel LQ theory does not establish global solvability of
the coupled pair.

This paper develops the following results for the finite-horizon mixed game:
\begin{itemize}
\item \textbf{Bounded-real characterization and mild verification:}
The strict disturbance gain condition is equivalent to unique strongly regular
mild Riccati solvability, by specialization of stochastic LQ uniform-convexity
theory. Simultaneous approximation of the Lyapunov equation and the state
justifies the quadratic identity used to verify the closed-loop Nash
inequalities and full-output strict attenuation.
\item \textbf{Local coupled solvability:}
An invertible feedback block near the zero terminal weights and a contraction
in the mild solution space yield a locally unique coupled Riccati pair for
every fixed $\gamma>0$. General prescribed-horizon solvability remains an
explicit hypothesis of the verification theorem.
\item \textbf{Global heat-equation example:}
For the specified coefficients at $\gamma=0.09$, a uniform invariant rectangle
proves existence on every finite horizon, bounded infinite-dimensional
feedbacks, and operator-norm spectral convergence. Numerical projections
illustrate the gains and selected best responses.
\end{itemize}

The remainder of this paper is organized as follows. Section II presents the basic operator model and preliminaries for infinite-dimensional stochastic control systems. Section III formulates the infinite-dimensional closed-loop mixed $H_2/H_\infty$ game problem. Section IV establishes the stochastic bounded real lemma. Section V verifies closed-loop Nash strategies and proves local coupled Riccati solvability. Section VI gives a heat-equation spectral-projection illustration. Section VII concludes the paper.

\section{Preliminaries}

This section fixes the functional setting, states the standing assumptions,
introduces the stochastic system, and recalls its well-posedness.
\subsection{Functional Setting and Notation}

Let $(\Omega,\mathcal F,\mathbb F,\mathbb P)$ be a complete filtered probability
space carrying a standard one-dimensional Brownian motion
$W(\cdot)=\{W(t)\}_{t\ge0}$. The filtration
$\mathbb F=\{\mathcal F_t\}_{t\ge0}$ is the natural filtration of $W$, augmented
by all $\mathbb P$-null sets in $\mathcal F$.

Fix a time horizon $T>0$. We use the following notation:

\begin{itemize}
\item $H,U,V,Z_1,Z_2$: Real separable Hilbert spaces for the state,
control, disturbance, and the two components of the performance output,
respectively.

\item $\mathcal L(X;Y)/\mathcal L(X)$: Spaces of bounded linear operators
from $X$ to $Y$ and from $X$ to itself, respectively. For
$F\in\mathcal L(X;Y)$, $F^*$ denotes its adjoint.

\item $\langle\cdot,\cdot\rangle_X/\|\cdot\|_X$: Inner product and norm
on $X$, respectively; the subscript is omitted when clear.

\item $\mathcal S(H)$: Space of bounded self-adjoint operators on $H$.
\item $W^{1,1}(t,T;E)$: Bochner--Sobolev space of $E$-valued functions
$f$ with $f,\dot f\in L^1(t,T;E)$ and an absolutely continuous
representative. For $E=\mathcal S(H)$, the norm is the operator norm.

\item $L^\infty(t,T;\mathcal L(X;Y))$: Strongly operator measurable
$\varphi:[t,T]\to\mathcal L(X;Y)$, meaning that $s\mapsto\varphi(s)x$
is measurable in $Y$ for every $x\in X$, with
$\operatorname*{ess\,sup}_{s\in[t,T]}\|\varphi(s)\|_{\mathcal L(X;Y)}<\infty$.
All Hilbert spaces are separable, so adjoint coefficients are strongly
operator measurable as well. This convention does not require Bochner
measurability in the operator norm.

\item $L^2(\Omega,\mathcal F_t,\mathbb P;X)$: Space of
$\mathcal F_t$-strongly measurable $X$-valued random variables $\xi$
with $\mathbb E\|\xi\|_X^2<\infty$.

\item $L^2_{\mathbb F}(t,T;X)$: Space of $\mathbb F$-predictable
$X$-valued processes $\varphi$ satisfying
$\mathbb E\int_t^T\|\varphi(s)\|_X^2\,ds<\infty$.

\item $C_{\mathbb F}([t,T];L^2(\Omega;X))$: Space of $\mathbb F$-adapted
processes $\varphi$ satisfying
$\varphi\in C([t,T];L^2(\Omega;X))$.
\end{itemize}

\subsection{Standing Assumptions and System Model}

We first specify the operator framework in which the state equation is posed.
\begin{assum} \label{assum:wellposed}
The operator $A:D(A)\subset H\to H$ generates a $C_0$-semigroup
$\{S(s)\}_{s\ge0}$ on $H$. The coefficient mappings are deterministic, strongly
measurable, and essentially bounded in the following spaces:
\[
\begin{aligned}
&A_1,A_2\in L^\infty(0,T;\mathcal L(H)),\\
&B_1,B_2\in L^\infty(0,T;\mathcal L(U;H)),\\
&C_1,C_2\in L^\infty(0,T;\mathcal L(V;H)),\\
&M\in L^\infty(0,T;\mathcal L(H;Z_1)).
\end{aligned}
\]
\end{assum}

\begin{assum} \label{assum:output}
The operator $N\in L^\infty(0,T;\mathcal L(U;Z_2))$ satisfies the control-weight
normalization
\[
N^*(s)N(s)=I_U,\qquad \text{for a.e. }s\in[0,T],
\]
where $I_U$ is the identity operator on $U$.
\end{assum}

For $t\in[0,T]$, define the admissible input spaces
$\mathcal U[t,T]\triangleq L^2_{\mathbb F}(t,T;U)$ and
$\mathcal V[t,T]\triangleq L^2_{\mathbb F}(t,T;V)$, and set
$Z\triangleq Z_1\times Z_2$. Given
$\xi\in L^2(\Omega,\mathcal F_t,\mathbb P;H)$,
$u\in\mathcal U[t,T]$, and $v\in\mathcal V[t,T]$, consider the
infinite-dimensional stochastic control system
\begin{equation} \label{eq:system_model}
\left\{
\begin{aligned}
dX(s)&=\big[(A+A_1(s))X(s)+B_1(s)u(s)\\
&\qquad+C_1(s)v(s)\big]\,ds\\
&\quad+\big[A_2(s)X(s)+B_2(s)u(s)\\
&\qquad+C_2(s)v(s)\big]\,dW(s),\\
z(s)&=\begin{pmatrix}M(s)X(s)\\N(s)u(s)\end{pmatrix},
\qquad X(t)=\xi .
\end{aligned}
\right.
\end{equation}
Its state is denoted by $X(\cdot;t,\xi,u(\cdot),v(\cdot))$, and its performance
output $z$ takes values in $Z$.

\begin{rem}
For the performance output in \eqref{eq:system_model},
Assumption~\ref{assum:output} gives
\[
\|z(s)\|_Z^2=\|M(s)X(s)\|_{Z_1}^2+\|u(s)\|_U^2.
\]
Thus, $N^*N=I_U$ simply normalizes the control-energy weight. If a model contains
a bounded coercive weight $N^*N=R(s)\ge rI_U$ instead, one may use the weighted
norm $\|u\|_{R(s)}^2=\langle R(s)u,u\rangle$ and replace the identity operators in
the $H_2$ terms by the corresponding weight. Hence the normalization
simplifies notation and imposes no essential restriction on the control cost.
\end{rem}

\subsection{Existence and Uniqueness of the Solution}

The following standard well-posedness result follows from the mild solution theory for stochastic evolution equations; see, for example, Lemma 3.1 in \cite{lu2019well} and the semigroup formulation in \cite{govindan2016yosida}.
\begin{thm}\label{thm:existence_uniqueness}
Suppose that Assumption \ref{assum:wellposed} holds. For any given initial condition $(t,\xi) \in [0,T] \times L^2(\Omega, \mathcal{F}_t, \mathbb{P}; H)$, control input $u(\cdot) \in \mathcal{U}[t,T]$, and external disturbance $v(\cdot) \in \mathcal{V}[t,T]$, the system \eqref{eq:system_model} admits a unique mild solution
\[
X(\cdot) \in C_{\mathbb{F}}([t,T]; L^2(\Omega; H)).
\]
Furthermore, there exists a constant $C > 0$, independent of $\xi$, $u(\cdot)$, and $v(\cdot)$, such that
\begin{equation}\label{eq:mild_estimate}
\begin{aligned}
&\|X(\cdot;t,\xi,u,v)\|^2_{C_{\mathbb F}([t,T];L^2(\Omega;H))}\\
&\quad\le C\bigl(\|\xi\|^2_{L^2(\Omega;H)}
+\|u\|^2_{\mathcal U[t,T]}+\|v\|^2_{\mathcal V[t,T]}\bigr).
\end{aligned}
\end{equation}
\end{thm}

\subsection{Mild Operator Solutions and Optional Time Regularity}

Write $C_{\mathcal S}([t,T];\mathcal S(H))$ for uniformly bounded,
self-adjoint operator functions $P$ for which $s\mapsto P(s)x$ is continuous
for every $x\in H$. Its metric is the uniform operator norm. This is a
complete metric space; uniform operator-norm convergence preserves statewise
continuity. Integrals in a mild operator identity are taken after application
to each $x\in H$. No operator-norm differentiability is assumed.

For comparison with strong verification arguments, let $\mathcal R[t,T]$
denote the subclass of $W^{1,1}(t,T;\mathcal S(H))$ satisfying the relevant
weak identity on $D(A)$ and having bounded measurable feedbacks.
Membership in this smaller class is optional in the results below.
Strong regularity always means a separately stated, uniform positive lower
bound for the effective input weight; it is not implied by time regularity.

Subsequent analysis addresses closed-loop strategies for the infinite-dimensional mixed $H_2/H_\infty$ problem. Following the stochastic mixed $H_2/H_\infty$ formulation in \cite{wang2023stochastic}, 
the design objective is a mutual-best-response feedback pair with strict disturbance attenuation. The control best response is evaluated against a fixed disturbance feedback; no global minimum over all feasible controllers and their recomputed disturbance responses is claimed.

\section{Infinite-Dimensional Mixed \texorpdfstring{$H_2/H_\infty$}{H2/H-infinity} Closed-Loop Game}
\label{sec:game}

\subsection{Admissible Strategy Sets}
First, we define the sets of admissible closed-loop control strategies, $\mathcal{N}^2[t,T]$, and disturbance strategies, $\mathcal{M}^2[t,T]$, as follows:
\begin{align*}
\mathcal{N}^2[t,T] &\triangleq L^\infty(t, T; \mathcal{L}(H; U)), \\
\mathcal{M}^2[t,T] &\triangleq L^\infty(t, T; \mathcal{L}(H; V)).
\end{align*}
We consider closed-loop control and disturbance strategies in the form of linear bounded operator feedback:
\[
u(s) = K(s)X(s), \quad v(s) = L(s)X(s),
\]
where $K(s)$ and $L(s)$ are feedback gain operators satisfying $K(\cdot) \in \mathcal{N}^2[t, T]$ and $L(\cdot) \in \mathcal{M}^2[t, T]$.

\subsection{Perturbation Operator and \texorpdfstring{$H_2/H_\infty$}{H2/H-infinity} Performance Indices}
In the framework of $H_\infty$ control theory, the robustness of a system aims to suppress the energy amplification effect of the worst-case external disturbance on the regulated output. Therefore, assuming a zero initial state (i.e., $X(t)=0$) and a given admissible control feedback law $u(\cdot) = K(\cdot)X(\cdot)$, we define the linear perturbation operator $\mathcal{L}_K$ from the disturbance input $v(\cdot)$ to the performance output $z(\cdot)$ as:
\[
\mathcal{L}_K : L_{\mathbb{F}}^2(t, T; V) \to L_{\mathbb{F}}^2(t, T; Z),
\]
satisfying $\mathcal{L}_K(v) = z$.

The core objective of infinite-dimensional $H_\infty$ control is to find an admissible control strategy such that the induced norm (i.e., the $L_2$-gain) of this operator is strictly less than a prescribed disturbance attenuation level $\gamma > 0$:
\[
\|\mathcal{L}_K\| = \sup_{\substack{v \neq 0 \\ v \in L_{\mathbb{F}}^2(t,T;V)}} \frac{\|z\|_{L_{\mathbb{F}}^2(t,T;Z)}}{\|v\|_{L_{\mathbb{F}}^2(t,T;V)}} < \gamma,
\]
where the norms in the space of square-integrable predictable stochastic processes are defined by
\begin{align*}
\|v\|_{L_{\mathbb F}^2(t,T;V)}^2
&=\mathbb E\int_t^T\|v(s)\|_V^2\,ds,\\
\|z\|_{L_{\mathbb F}^2(t,T;Z)}^2
&=\mathbb E\int_t^T\bigl(\|M(s)X(s)\|_{Z_1}^2+\|u(s)\|_U^2\bigr)\,ds.
\end{align*}
For a prescribed attenuation level $\gamma > 0$, the $H_\infty$ robustness
cost functional $J_1$ penalizes disturbance energy and rewards output energy:

\[
\begin{aligned}
J_1(t,\xi;u,v)=\mathbb E\int_t^T
\bigl[&\gamma^2\|v(s)\|_V^2-\|M(s)X(s)\|_{Z_1}^2\\
&-\|u(s)\|_U^2\bigr]\,ds.
\end{aligned}
\]

Similarly, given the disturbance feedback strategy $v(s)=L(s)X(s)$, define the finite-horizon quadratic performance functional
\[
\begin{aligned}
J_2(t,\xi;u,v)=\mathbb E\int_t^T
\bigl[&\|M(s)X(s)\|_{Z_1}^2+\|u(s)\|_U^2\bigr]\,ds.
\end{aligned}
\]

For a fixed control feedback $K$, the disturbance minimizes the
$\gamma$-penalized functional $J_1$. For a fixed disturbance feedback $L$,
the controller minimizes $J_2$ over admissible predictable controls. We
\emph{define} the mixed design as the resulting nonzero-sum feedback Nash
problem; equivalence to a separate constrained global optimization problem
is not asserted.

\begin{rem}
The term ``worst-case disturbance'' below refers to a minimizing response for
the $\gamma$-penalized criterion, not to an input attaining the induced gain.
With zero initial state and strict attenuation, $J_1(0;KX,v)$ is uniformly
positive for $v\ne0$, so its minimizing response is $v=0$. The induced gain,
in contrast, is a supremum of an energy ratio over nonzero disturbances.
Moreover, the Nash inequality fixes $L$ during a control deviation; it does
not compare a competing controller with its newly computed disturbance
best response. The $H_2$ criterion is the LQ energy in this game and does not
assert nominal constrained $H_2$ optimality.
\end{rem}

\subsection{Closed-Loop Nash Equilibrium Strategy}

\begin{defn} \label{def:nash_equilibrium}
A pair of admissible closed-loop strategies $(K(\cdot),L(\cdot))$ is called a
closed-loop Nash equilibrium strategy if, for every initial pair
$(t,\xi)\in[0,T)\times L^2(\Omega,\mathcal F_t,\mathbb P;H)$, the associated state
$X^*$ and outcomes $u^*(s)=K(s)X^*(s)$, $v^*(s)=L(s)X^*(s)$ satisfy the following
two conditions:

(i) $H_\infty$ robustness criterion: Given the fixed control feedback operator $K(\cdot)$, for any predictable disturbance process $v(\cdot) \in L_{\mathbb{F}}^2(t, T; V)$, there holds
\begin{equation}\label{eq:nash_hinf}
J_1\Big(t, \xi; u^*(\cdot), v^*(\cdot)\Big) \le J_1\Big(t, \xi; K(\cdot)X(\cdot), v(\cdot)\Big),
\end{equation}
where $X(\cdot) \in C_{\mathbb{F}}([t,T]; L^2(\Omega;H))$ is the state trajectory generated by the fixed control feedback $K(\cdot)$ and the arbitrary disturbance $v(\cdot)$.

\vspace{1.5mm}
(ii) $H_2$ optimal regulation criterion: Given the fixed disturbance feedback operator $L(\cdot)$, for any predictable control process $u(\cdot) \in L_{\mathbb{F}}^2(t, T; U)$, there holds
\begin{equation}\label{eq:nash_h2}
J_2\Big(t, \xi; u^*(\cdot), v^*(\cdot)\Big) \le J_2\Big(t, \xi; u(\cdot), L(\cdot)X(\cdot)\Big),
\end{equation}
where $X(\cdot)$ is the state trajectory generated by the arbitrary control $u(\cdot)$ and the fixed disturbance feedback $L(\cdot)$.

In this case, the process pair $(u^*(\cdot), v^*(\cdot))$ is referred to as the closed-loop Nash equilibrium point.
\end{defn}

 \section{Stochastic Bounded Real Lemma}
 \label{sec:sbrl}

This section first proves the bounded real verification argument for the disturbance channel. By setting $u=0$ in system \eqref{eq:system_model}, we obtain an infinite-dimensional linear stochastic differential equation over a finite time horizon. For a given initial condition $(t,\xi) \in [0,T] \times L^2(\Omega, \mathcal{F}_t, \mathbb{P}; H)$, the system evolves as follows:
\begin{equation}\label{eq:disturbance_system}
\begin{cases}
\begin{aligned}
dX(s)&=\big((A+A_1)X(s)+C_1(s)v(s)\big)\,ds\\
&\quad+\big(A_2X(s)+C_2(s)v(s)\big)\,dW(s),
\end{aligned}\\
Z(s)=M(s)X(s),\qquad s\in[t,T],\\
X(t) = \xi.
\end{cases}
\end{equation}
This equation describes the evolution of the state $X(s)$ under the influence of the stochastic disturbance $v(s)$. Under the zero initial condition $X(t)=0$, we define the linear perturbation operator $\mathcal{L}: L_{\mathbb{F}}^2(t, T; V) \to L_{\mathbb{F}}^2(t, T; Z_1)$ for system \eqref{eq:disturbance_system} by
\[
(\mathcal{L}v)(s) = Z(s) = M(s)X(s).
\]
The induced norm of the operator $\mathcal{L}$ is defined as the $H_\infty$ norm of the system:
\[
\|\mathcal{L}\| = \sup_{\substack{v \in L_{\mathbb{F}}^2(t,T; V) \\ v \neq 0}} \frac{\|\mathcal{L}v\|_{L_{\mathbb{F}}^2(t,T; Z_1)}}{\|v\|_{L_{\mathbb{F}}^2(t,T; V)}}.
\]
Thus, $\|\mathcal L\|$ is the worst-case induced $L^2$ gain from the disturbance
$v(\cdot)$ to the output $Z(\cdot)$. A smaller gain means that less output energy
can be generated per unit of disturbance energy and therefore indicates stronger
disturbance attenuation.  For an
arbitrary initial pair
$(t,\xi)\in[0,T]\times L^2(\Omega,\mathcal F_t,\mathbb P;H)$, the worst-case
disturbance is characterized by the minimization problem
\begin{align*}
&\min_{v\in L_{\mathbb F}^2(t,T;V)}J_1(t,\xi;0,v)\\
&=\min_{v\in L_{\mathbb F}^2(t,T;V)}\mathbb E\int_t^T
 \bigl[\gamma^2\|v(s)\|_V^2-\|Z(s)\|_{Z_1}^2\bigr]\,ds\\
&=\min_{v\in L_{\mathbb F}^2(t,T;V)}\mathbb E\int_t^T
 \bigl[\langle\gamma^2I_Vv(s),v(s)\rangle\\
&\hspace{3.0cm}-\langle M^*(s)M(s)X(s),X(s)\rangle\bigr]\,ds.
\end{align*}

For the sake of notational simplicity, the time variable $s$ in the operators will be omitted in the sequel when no confusion arises. Since $A$ is unbounded, the following Riccati expressions are formal differential notations. They are used only through their weak form on $D(A)$ or through the corresponding mild identities obtained by the semigroup/Yosida approximation procedure. In particular, products such as $P(A+A_1)$ and $(A+A_1)^*P$ are not regarded as bounded operators on $H$ unless additional regularity is imposed.
\[
\begin{aligned}
\mathcal S(P)
&=\dot P+P(A+A_1)+(A+A_1)^*P\\
&\quad+A_2^*PA_2-M^*M,\\
\Lambda^\gamma(P)&=\gamma^2 I_V+C_2^*PC_2,\\
\Phi(P)&=PC_1+A_2^*PC_2.
\end{aligned}
\]
The formal completion-of-squares calculation leads to the following candidate operator Riccati equation:
\begin{equation}\label{eq:hinf_riccati}
\begin{cases}
\mathcal{S}(P) - \Phi(P)\big(\Lambda^\gamma(P)\big)^{-1}\Phi^*(P) = 0, \\
\Lambda^\gamma(P) \ge \delta I_V, \quad \delta > 0, \\
P(T) = 0.
\end{cases}
\end{equation}
For the necessary-and-sufficient bounded-real statement, we use the standard
mild-solution class from infinite-dimensional stochastic LQ theory. Denote by
$C_{\mathcal S}([t,T];\mathcal S(H))$ the set of uniformly bounded self-adjoint
operator-valued mappings $P$ such that $s\mapsto P(s)x$ is continuous in $H$ for
every $x\in H$.

\begin{defn} \textit{(Strongly regular mild Riccati solution)}
\label{def:strongly_regular_mild}
A mapping $P\in C_{\mathcal S}([t,T];\mathcal S(H))$ is called a mild solution of
\eqref{eq:hinf_riccati} if, for every $x\in H$ and $s\in[t,T]$,
\begin{equation}\label{eq:hinf_riccati_mild}
P(s)x=\int_s^T S(r-s)^*\mathfrak R(P(r))S(r-s)x\,dr,
\end{equation}
where the integral is understood in $H$ and
\[
\begin{aligned}
\mathfrak R(P)
&=PA_1+A_1^*P+A_2^*PA_2-M^*M\\
&-\Phi(P)\big(\Lambda^\gamma(P)\big)^{-1}\Phi(P)^*.
\end{aligned}
\]
It is called strongly regular if there is a constant $\delta>0$ such that
\[
\Lambda^\gamma(P(s))\ge\delta I_V\qquad\text{for a.e. }s\in[t,T].
\]
This is the specialization of the strongly regular mild solution in
\cite{lu2019well} to $B=C_1$, $C=A_2$, $D=C_2$, $Q=-M^*M$,
$R=\gamma^2I_V$, and terminal weight $G=0$.
\end{defn}

For formal block calculations on $D(A)\times V$, we denote
\begin{equation}\label{eq:psi_p}
\begin{aligned}
Q(P)&=\begin{pmatrix}
\mathcal S(P)&\Phi(P)\\
\Phi^*(P)&\Lambda^\gamma(P)
\end{pmatrix},\\
\Psi_P(s)&=-\bigl(\Lambda^\gamma(P(s))\bigr)^{-1}
\Phi^*(P(s)).
\end{aligned}
\end{equation}

\begin{lem} \textit{(Quadratic identity for a mild Lyapunov solution)}
\label{lem:general_mild_identity}
Let $F,G$ be bounded deterministic operator coefficients, and let $Q$ be
bounded, self-adjoint and strongly operator measurable. The terminal-zero
Lyapunov equation
\[
\dot D+D(A+F)+(A+F)^*D+G^*DG+Q=0,\qquad D(T)=0
\]
has a unique mild solution $D\in C_{\mathcal S}([t,T];\mathcal S(H))$.
For $\xi\in L^2(\Omega,\mathcal F_t;H)$ and predictable $f,g\in
L^2_{\mathbb F}(t,T;H)$, let
\[
dX=((A+F)X+f)\,ds+(GX+g)\,dW,\qquad X(t)=\xi.
\]
Then
\begin{equation}\label{eq:general_mild_identity}
\begin{aligned}
\mathbb E\int_t^T\langle QX,X\rangle\,ds
&=\mathbb E\langle D(t)\xi,\xi\rangle\\
&\quad+\mathbb E\int_t^T\bigl(2\langle DX,f\rangle
+2\langle DGX,g\rangle\bigr)\,ds\\
&\quad+\mathbb E\int_t^T\langle Dg,g\rangle\,ds.
\end{aligned}
\end{equation}
In particular, this identity requires neither $D\in\mathcal R[t,T]$ nor
a direct It\^o formula for a mild state.
\end{lem}
\begin{pf}
The mild integral map is a contraction on a sufficiently short terminal
interval in the uniform operator norm: its linear part has Lipschitz
constant at most $C h(2\|F\|_\infty+\|G\|_\infty^2)$ on an interval of
length $h$. Backward continuation yields existence, uniqueness and a
uniform bound on the full finite interval.

Let $A_n=nA(nI-A)^{-1}$, for sufficiently large $n$, and let $S_n$ be its
semigroup. Yosida approximation gives a common semigroup bound on $[0,T]$
and strong convergence of $S_n$ and $S_n^*$, uniform on compact time
intervals. Let $D_n$ solve the \emph{Lyapunov equation with generator $A_n$}
and the same $F,G,Q$, and let $X_n$ solve the state equation with $A_n$ and
the same $\xi,f,g$. Thus both the operator equation and the state are
approximated; we do not apply a weak derivative of $D$ to a random vector.

For each fixed $n$, the mild equation for $D_n$ implies norm Lipschitz
continuity and a strongly operator measurable bounded derivative satisfying
the bounded-generator equation. Extend $D_n$ constantly outside $[t,T]$
and mollify in time. The mollified operators are smooth in operator norm,
converge uniformly to $D_n$, and their derivatives converge strongly a.e.
to its strong derivative, with a bound depending on $n$. The last assertion
follows first on a countable dense subset of $H$ from vector-valued Lebesgue
differentiation and then on all of $H$ from the uniform derivative bound.
Apply the classical Hilbert-space It\^o formula to the mollified quadratic
forms and $X_n$ and let the mollification parameter tend to zero. This
proves \eqref{eq:general_mild_identity} with $D_n,X_n$.

The expectation step can first be performed for bounded $\xi$ and bounded
predictable $f,g$. For fixed $n$, the corresponding state has fourth
moments on the finite horizon, so the quadratic stochastic integral is
integrable after the usual localization. General $L^2$ data follow by
truncation: the state estimate and Cauchy--Schwarz show continuity of every
term in \eqref{eq:general_mild_identity}. No fourth-moment assumption is
thereby imposed on the original data.

For completeness, the operator convergence needed in the last limit is
strong convergence, not convergence in operator norm. Expand the mild
Lyapunov equation in its Volterra series. Its $m$th linear iterate is
bounded by $C_0(C_1(T-t))^m/m!$, uniformly in $n$. In each fixed iterate,
the strong convergence of $S_n,S_n^*$ and dominated convergence give
$D_n(s)x\to D(s)x$ for every $s,x$. The uniformly summable bounds permit
passage from the finite iterates to the full solution and also give
$\sup_{n,s}\|D_n(s)\|<\infty$.
The state approximation satisfies
$X_n\to X$ in $C([t,T];L^2(\Omega;H))$: this follows by the mild
equations, It\^o isometry, the common semigroup bound and Gronwall's
inequality, first for simple data and then by their $L^2$ density.
Strong convergence of uniformly bounded $D_n(s)$ extends to random
$L^2$ vectors by simple-vector approximation. Dominated convergence and
Cauchy--Schwarz therefore pass each term of the bounded-generator
identity to the limit. This proves \eqref{eq:general_mild_identity}.
\end{pf}

\begin{lem} \textit{(Mild Riccati quadratic identity)} \label{lem:j1_expansion}
Let $P$ be a strongly regular mild solution of \eqref{eq:hinf_riccati}. Let $X(\cdot)$ be the mild solution of \eqref{eq:disturbance_system} corresponding to $\xi\in L^2(\Omega,\mathcal F_t,\mathbb P;H)$ and $v(\cdot)\in\mathcal V[t,T]$. Then the following quadratic identity holds:
\begin{equation}\label{eq:j1_identity}
\begin{aligned}
J_1(t, \xi; 0, v(\cdot))
&= \mathbb{E}\langle \xi, P(t)\xi \rangle- \mathbb{E}\langle X(T), P(T)X(T) \rangle \\
&\hspace{-0.3cm}\quad + \mathbb{E} \int_t^T\Big\langle\Lambda^\gamma(P(s))
\big(v(s)-\Psi_P(s)X(s)\big),\\
&\quad v(s)-\Psi_P(s)X(s) \Big\rangle_V\,ds.
\end{aligned}
\end{equation}
\end{lem}
\begin{pf}
Apply Lemma~\ref{lem:general_mild_identity} with $F=A_1$, $G=A_2$,
$f=C_1v$, $g=C_2v$, and
\[
Q=-M^*M-\Phi(P)(\Lambda^\gamma(P))^{-1}\Phi(P)^*.
\]
The mild Riccati equation is exactly the Lyapunov equation with this
bounded strongly measurable forcing. Add the disturbance energy to the
identity and complete the square using
$\Lambda^\gamma(P)\Psi_P=-\Phi(P)^*$. This gives
\eqref{eq:j1_identity}; $P(T)=0$ in the present problem.
\end{pf}

\begin{proposition} \textit{(Feedback input transformation)}
\label{prop:feedback_input_isomorphism}
Suppose Assumption~\ref{assum:wellposed} holds and let
$\Psi\in L^\infty(t,T;\mathcal L(H;V))$. For $v\in\mathcal V[t,T]$, let
$X^v=X(\cdot;t,0,0,v)$ be the mild solution of
\eqref{eq:disturbance_system}, and define
\[
(\Gamma_\Psi v)(s)=v(s)-\Psi(s)X^v(s).
\]
Then $\Gamma_\Psi$ is a bounded linear isomorphism of $\mathcal V[t,T]$.
More precisely, for every $h\in\mathcal V[t,T]$, let $X_h$ be the unique mild
solution of
\begin{equation}\label{eq:feedback_inverse_state}
\begin{cases}
\begin{aligned}
dX_h(s)&=\big[(A+A_1+C_1\Psi)X_h(s)+C_1h(s)\big]ds\\
&\quad+\big[(A_2+C_2\Psi)X_h(s)+C_2h(s)\big]dW(s),
\end{aligned}\\
X_h(t)=0.
\end{cases}
\end{equation}
Then
\[
(\Gamma_\Psi^{-1}h)(s)=h(s)+\Psi(s)X_h(s).
\]
Consequently, there are constants $m_\Psi,M_\Psi>0$ such that
\begin{equation}\label{eq:feedback_isomorphism_bounds}
m_\Psi\|v\|_{\mathcal V[t,T]}
\le\|\Gamma_\Psi v\|_{\mathcal V[t,T]}
\le M_\Psi\|v\|_{\mathcal V[t,T]}.
\end{equation}
\end{proposition}

\begin{pf}
Linearity is immediate. By Theorem~\ref{thm:existence_uniqueness},
\[
\|X^v\|_{L^2_{\mathbb F}(t,T;H)}\le C\|v\|_{\mathcal V[t,T]},
\]
and hence $\Gamma_\Psi$ is bounded. Given $h\in\mathcal V[t,T]$, the
coefficients $A_1+C_1\Psi$ and $A_2+C_2\Psi$ in
\eqref{eq:feedback_inverse_state} are bounded and strongly measurable. Thus the
standard mild well-posedness theorem, applied with the original semigroup
generated by $A$, gives a unique $X_h$ and
\[
\|X_h\|_{L^2_{\mathbb F}(t,T;H)}\le C_\Psi\|h\|_{\mathcal V[t,T]}.
\]
Define $v=h+\Psi X_h$. Substitution into the original disturbance equation shows
that its state $X^v$ satisfies \eqref{eq:feedback_inverse_state}; uniqueness gives
$X^v=X_h$ and hence $\Gamma_\Psi v=h$. This proves surjectivity. If
$\Gamma_\Psi v=0$, the same argument gives the homogeneous closed-loop equation
with zero initial state, so $X^v=0$ and $v=0$. Thus $\Gamma_\Psi$ is injective.
Finally,
\[
\|\Gamma_\Psi^{-1}h\|_{\mathcal V[t,T]}
\le\big(1+\|\Psi\|_\infty C_\Psi\big)\|h\|_{\mathcal V[t,T]},
\]
which proves boundedness of the inverse and
\eqref{eq:feedback_isomorphism_bounds}. In particular one may take
$m_\Psi=\|\Gamma_\Psi^{-1}\|^{-1}>0$.
\end{pf}

We next establish the estimates needed to construct a Riccati solution from the
strict-gain condition. The first identity is stated for mild Lyapunov solutions so
that it can be applied to every iterate in the operator-valued approximation below.

\begin{lem} \textit{(Mild Lyapunov representation and uniform operator estimate)}
\label{lem:mild_lyapunov_uniform_estimate}
Let $F,G\in L^\infty(t,T;\mathcal L(H))$ and
$Q\in L^\infty(t,T;\mathcal S(H))$. The terminal-value Lyapunov equation
\begin{equation}\label{eq:abstract_mild_lyapunov}
\begin{cases}
\dot D+D(A+F)+(A+F)^*D+G^*DG+Q=0,\\
D(T)=0,
\end{cases}
\end{equation}
has a unique mild solution $D\in C_{\mathcal S}([t,T];\mathcal S(H))$. If
$Y^{s,x}$ is the mild solution of
\[
\begin{aligned}
dY(r)&=(A+F(r))Y(r)\,dr+G(r)Y(r)\,dW(r),\\
Y(s)&=x,
\end{aligned}
\]
then
\begin{equation}\label{eq:positive_lyapunov_representation}
\langle D(s)x,x\rangle
=\mathbb E\int_s^T\langle Q(r)Y^{s,x}(r),Y^{s,x}(r)\rangle\,dr.
\end{equation}
In particular, $Q\ge0$ implies $D\ge0$, and there is a constant
$C_L=C_L(T,S,\|F\|_\infty,\|G\|_\infty)$ such that
\begin{equation}\label{eq:uniform_lyapunov_estimate}
\|D(s)\|_{\mathcal L(H)}
\le C_L\int_s^T\|Q(r)\|_{\mathcal L(H)}\,dr,
\qquad s\in[t,T].
\end{equation}
For a family $(F_j,G_j)$ bounded uniformly in $L^\infty$, the same $C_L$ may be
used for every $j$.
\end{lem}

\begin{pf}
Existence and uniqueness follow from
Lemma~\ref{lem:general_mild_identity}. Apply that lemma with $f=g=0$
to the homogeneous state $Y^{s,x}$; this proves
\eqref{eq:positive_lyapunov_representation} directly for mild solutions.

The standard homogeneous mild-state estimate gives
\[
\sup_{r\in[s,T]}\mathbb E\|Y^{s,x}(r)\|_H^2
\le C_L\|x\|_H^2.
\]
Hence \eqref{eq:positive_lyapunov_representation} yields
\[
|\langle D(s)x,x\rangle|
\le C_L\|x\|_H^2\int_s^T\|Q(r)\|\,dr.
\]
Because $D(s)$ is self-adjoint, taking the supremum over $\|x\|_H=1$ proves
\eqref{eq:uniform_lyapunov_estimate}. The mild-state estimate depends on $F$ and
$G$ only through their essential-supremum bounds, which proves the uniform-family
statement.
\end{pf}

\begin{lem} \textit{(Closed-loop perturbation identity)} \label{lem:closed_loop_perturbation}
For a given $\varphi\in L^\infty(t,T;\mathcal L(H;V))$, let $P\in C_{\mathcal S}([t,T];\mathcal S(H))$ be the mild solution to the closed-loop Lyapunov equation:
\begin{equation}\label{eq:closed_loop_lyapunov_mild}
P(s)x=\int_s^T S(r-s)^*\mathcal K_\varphi(P(r))S(r-s)x\,dr, \quad x\in H,
\end{equation}
where
\[
\begin{aligned}
& \mathcal K_\varphi(P) = P(A_1+C_1\varphi) +(A_1+C_1\varphi)^*P \\
& \quad +(A_2+C_2\varphi)^*P(A_2+C_2\varphi) -M^*M+\gamma^2\varphi^*\varphi.
\end{aligned}
\]
Equivalently, its weak differential form is:
\[
(I,\varphi^*(s))Q(P(s)) \begin{pmatrix}I\\ \varphi(s)\end{pmatrix}=0, \qquad P(T)=0.
\]
For any $v(\cdot)\in\mathcal V[t,T]$, define the perturbed input $v_\varphi(s)=v(s)+\varphi(s)X(s)$, with $X$ being the mild solution to:
\[
\begin{cases}
\begin{aligned}
& dX(s)=\big[(A+A_1+C_1\varphi)X(s)+C_1v(s)\big]ds\\
& \qquad\quad +\big[(A_2+C_2\varphi)X(s)+C_2v(s)\big]dW(s),
\end{aligned}\\
X(t)=\xi.
\end{cases}
\]
Then, the cost functional satisfies:
\[
\begin{aligned}
& J_1(t,\xi;0,v_\varphi) = \mathbb E\langle \xi,P(t)\xi\rangle +\mathbb E\int_t^T \big[2\langle v(s),\mathfrak N_\varphi(s)X(s)\rangle\\
& \quad  +\langle v(s),\Lambda^\gamma(P(s))v(s)\rangle\big]ds,
\end{aligned}
\]
where $\mathfrak N_\varphi(s)=\Phi^*(P(s))+\Lambda^\gamma(P(s))\varphi(s)$. Specifically, $J_1(t,\xi;0,\varphi X)=\mathbb E\langle \xi,P(t)\xi\rangle$ when $v\equiv 0$.
\end{lem}

\begin{pf}
Apply Lemma~\ref{lem:general_mild_identity} with
$F=A_1+C_1\varphi$, $G=A_2+C_2\varphi$,
$Q=-M^*M+\gamma^2\varphi^*\varphi$, $f=C_1v$ and $g=C_2v$.
Expand $\gamma^2\|v+\varphi X\|^2-\|MX\|^2$ and collect the
cross terms. The coefficient of the term linear in $v$ is precisely
$\Phi(P)^*+\Lambda^\gamma(P)\varphi$, and the quadratic coefficient
is $\Lambda^\gamma(P)$. This proves the claimed identity without requiring
operator-norm time differentiability of $P$.
\end{pf}

\begin{lem} \textit{(Lower bound under strict gain)} \label{lem:j1_lower_bound}
Assume that Assumption~\ref{assum:wellposed} holds and
$\|\mathcal L\|<\gamma$. Then there exists $\mu>0$,
independent of the starting time $\tau\in[t,T]$, such that, for every
$\xi\in L^2(\Omega,\mathcal F_\tau,\mathbb P;H)$ and
$v\in\mathcal V[\tau,T]$,
\[
J_1(\tau,\xi;0,v)\ge -\mu\,\mathbb E\|\xi\|_H^2 .
\]
\end{lem}

\begin{pf}
Fix $\tau\in[t,T]$. Extending an input on $[\tau,T]$ by zero on $[t,\tau)$
shows that the same strict-gain gap is valid for the subsystem starting at $\tau$.
We write $t=\tau$ in the calculation below.
Choose $\varepsilon>0$ with $\varepsilon^2\le\gamma^2-\|\mathcal L\|^2$. For zero
initial state, the strict gain condition gives

{\small
\[
\begin{aligned}
J_1(t,0;0,v)
&=\mathbb E\int_t^T\bigl(\gamma^2\|v(s)\|_V^2-\|M(s)X^{t,0;v}(s)\|_{Z_1}^2\bigr)\,ds\\
&\ge \varepsilon^2\|v\|_{\mathcal V[t,T]}^2 .
\end{aligned}
\]
}

Let $P_0$ be the unique mild solution supplied by
Lemma~\ref{lem:mild_lyapunov_uniform_estimate} for the Lyapunov identity obtained from
Lemma \ref{lem:closed_loop_perturbation} with $\varphi=0$:
\[
\begin{aligned}
&0={}\dot P_0+P_0(A+A_1)+(A+A_1)^*P_0+A_2^*P_0A_2-M^*M,\\
&P_0(T)=0 .
\end{aligned}
\]
The corresponding quadratic identity gives
\[
\begin{aligned}
&J_1(t,\xi;0,v)-J_1(t,0;0,v)\\
&\quad=2\mathbb E\int_t^T
\langle v(s),\Phi^*(P_0(s))X^{t,\xi;0}(s)\rangle\,ds+\mathbb E\langle\xi,P_0(t)\xi\rangle .
\end{aligned}
\]
Here we used linearity:
$X^{t,\xi;v}=X^{t,0;v}+X^{t,\xi;0}$.
Consequently,
\[
\begin{aligned}
J_1(t,\xi;0,v)
&\ge \mathbb E\langle\xi,P_0(t)\xi\rangle
\!+\!\mathbb E\int_t^T
\varepsilon^2\|v(s)\|_V^2\,ds\\
&\quad+2\mathbb E\int_t^T
\big\langle v(s),\Phi^*(P_0(s))X^{t,\xi;0}(s)\big\rangle\,ds\\
&\ge \mathbb E\langle\xi,P_0(t)\xi\rangle
\! -\!\varepsilon^{-2}\mathbb E\int_t^T
\big\|\Phi^*(P_0(s))\\
&\hspace{2.0cm}\cdot X^{t,\xi;0}(s)\big\|_V^2\,ds .
\end{aligned}
\]
The homogeneous mild equation satisfies the standard estimate
\[
\begin{aligned}
\mathbb E\int_t^T\|X^{t,\xi;0}(s)\|_H^2\,ds
&\le C_0\mathbb E\|\xi\|_H^2 .
\end{aligned}
\]
Since $P_0$, $C_1$, $A_2$, and $C_2$ are bounded on $[t,T]$, there is $C_1'>0$ such
that
\[
\begin{aligned}
\mathbb E\int_t^T
\|\Phi^*(P_0(s))X^{t,\xi;0}(s)\|_V^2ds
&\le C_1'\mathbb E\|\xi\|_H^2 .
\end{aligned}
\]
Also $P_0$ is bounded, so
$\mathbb E\langle\xi,P_0(t)\xi\rangle\ge-\|P_0\|_\infty\mathbb E\|\xi\|_H^2$.
Taking $\mu=\|P_0\|_\infty+\varepsilon^{-2}C_1'$ proves the claim.
\end{pf}

\begin{lem} \textit{Coercivity from the strict gain estimate.} \label{lem:lambda_coercivity}
Suppose $P$ and $\varphi$ satisfy the conditions of Lemma \ref{lem:closed_loop_perturbation}. Under the strict bounded-real condition $\|\mathcal L\|<\gamma$, for any $\varepsilon>0$ satisfying $\varepsilon^2\le \gamma^2-\|\mathcal L\|^2$, the coercivity condition
\[
\Lambda^\gamma(P(s))\ge \varepsilon^2 I_V, \qquad \text{a.e. } s\in[t,T],
\]
holds with a common exceptional null set independent of the input vector.
\end{lem}

\begin{pf}
Fix $t_0\in[t,T)$ and $v_0\in V$. For a small $h>0$, consider the admissible deterministic spike input $v(s)=v_0{\bf 1}_{[t_0,t_0+h]}(s)$. Defining the perturbed input $v_\varphi=v+\varphi X$ as in Lemma~\ref{lem:closed_loop_perturbation}, where $X(t_0)=0$, the strict gain bound implies:\vspace{-1.5ex}
\[
J_1(t_0,0;0,v_\varphi) \ge \varepsilon^2\mathbb E\int_{t_0}^T\|v_\varphi(s)\|_V^2ds .
\]
Conversely, applying Lemma \ref{lem:closed_loop_perturbation} and noting that $v$ vanishes outside $[t_0,t_0+h]$, we truncate the non-negative tail energy $\int_{t_0+h}^T\|\varphi(s)X(s)\|^2 ds$ in the lower bound to obtain:\vspace{1ex}
\[
\begin{aligned}
&\varepsilon^2\mathbb E\int_{t_0}^{t_0+h} \|v_0+\varphi(s)X(s)\|_V^2ds \\
&\le \mathbb E\int_{t_0}^{t_0+h} \big[2\langle v_0,\mathfrak N_\varphi(s)X(s)\rangle +\langle v_0,\Lambda^\gamma(P(s))v_0\rangle\big]ds .
\end{aligned}
\]
For this spike input, standard mild formulation estimates and It\^o's isometry yield $\sup_{s\in[t_0,t_0+h]}\mathbb E\|X(s)\|_H^2 \le Ch\|v_0\|_V^2$ and $\|\mathbb E X(s)\|_H\le C(s-t_0)\|v_0\|_V$. Therefore,
\[
\begin{aligned}
&\frac1h\mathbb E\int_{t_0}^{t_0+h} \langle v_0,\mathfrak N_\varphi(s)X(s)\rangle ds \to 0,\\
&\frac1h\mathbb E\int_{t_0}^{t_0+h} \|\varphi(s)X(s)\|_V^2 ds \to 0.
\end{aligned}
\]
Dividing the truncated inequality by $h$ and taking the limit $h\downarrow 0$, the Lebesgue differentiation theorem guarantees:
\[
\varepsilon^2\|v_0\|_V^2 \le \langle v_0,\Lambda^\gamma(P(t_0))v_0\rangle.
\]
Perform this argument first on a countable dense subset of $V$, taking
the common set of Lebesgue points of the scalar quadratic forms. Its
complement is null. Boundedness of $\Lambda^\gamma(P(t_0))$ then extends
the inequality to every $v_0\in V$, proving the operator lower bound.
\end{pf}

\begin{thm} \textit{(Stochastic bounded real lemma)} \label{thm:sbrl}
Suppose Assumption~\ref{assum:wellposed} holds, fix $t\in[0,T)$, and let
$\gamma>0$. The following statements are equivalent:
\begin{enumerate}
\item the disturbance operator of \eqref{eq:disturbance_system} satisfies
$\|\mathcal L\|<\gamma$ on $[t,T]$;
\item the Riccati equation \eqref{eq:hinf_riccati} admits a unique strongly
regular mild solution $P$ in the sense of
Definition~\ref{def:strongly_regular_mild}.
\end{enumerate}
In this case, $\Psi_P=-(\Lambda^\gamma(P))^{-1}\Phi(P)^*$ belongs to
$L^\infty(t,T;\mathcal L(H;V))$. For every initial state $\xi$, the indefinite
LQ problem has the unique optimal disturbance
\[
v^*(s)=\Psi_P(s)X^*(s),
\]
and its value is $\mathbb E\langle\xi,P(t)\xi\rangle$.
\end{thm}

\begin{pf}
We first prove (ii)$\Rightarrow$(i). Let $P$ be a strongly regular mild solution
and put $\varphi=\Psi_P$. Expanding the identity
$\Phi^*(P)+\Lambda^\gamma(P)\Psi_P=0$ shows that the mild Riccati equation is
equivalent to the closed-loop Lyapunov equation
\eqref{eq:closed_loop_lyapunov_mild} with this $\varphi$. Hence
Lemma~\ref{lem:closed_loop_perturbation} gives, for zero initial state,
\[
\begin{aligned}
J_1(t,0;0,v)&=\mathbb E\int_t^T
\bigl\langle\Lambda^\gamma(P(s))
\big(v(s)-\Psi_P(s)X(s)\big),\\
&\hspace{1.8cm}v(s)-\Psi_P(s)X(s)\bigr\rangle\,ds .
\end{aligned}
\]
Because $P$ and the coefficients are bounded and
$\Lambda^\gamma(P)\ge\delta I_V$, the feedback $\Psi_P$ is essentially bounded.
Proposition~\ref{prop:feedback_input_isomorphism}, applied with
$\Psi=\Psi_P$, shows that the input transformation
$\Gamma v=v-\Psi_PX^{t,0;v}$ is a bounded isomorphism. With the explicit choice
$\beta=\|\Gamma^{-1}\|^{-2}>0$, we have
\[
\|v-\Psi_PX\|_{\mathcal V[t,T]}^2
=\|\Gamma v\|_{\mathcal V[t,T]}^2
\ge\beta\|v\|_{\mathcal V[t,T]}^2.
\]
Substitution into the $J_1$ identity gives
\[
\begin{aligned}
J_1(t,0;0,v)
&=\mathbb E\int_t^T
\bigl[\gamma^2\|v(s)\|^2-\|Z(s)\|^2\bigr]\,ds\\
&=\mathbb E\int_t^T\Big\langle
\Lambda^\gamma(P(s))\big(v(s)-\Psi_P(s)X(s)\big),\\
&\hspace{1.8cm}v(s)-\Psi_P(s)X(s)\Big\rangle\,ds\\
&\ge\delta\mathbb E\int_t^T
\|v(s)-\Psi_P(s)X(s)\|^2\,ds\\
&\ge\delta\beta\mathbb E\int_t^T\|v(s)\|^2\,ds.
\end{aligned}
\]
Rearranging this inequality yields
\[
\begin{aligned}
\mathbb E\int_t^T\|Z(s)\|^2\,ds
&\le(\gamma^2-\delta\beta)
\mathbb E\int_t^T\|v(s)\|^2\,ds.
\end{aligned}
\]
Since $\delta>0$ and $\beta=\|\Gamma^{-1}\|^{-2}>0$, the induced norm satisfies
$\|\mathcal{L}\| \leq (\gamma^2-\delta\beta)^{1/2}<\gamma$.

We now prove (i)$\Rightarrow$(ii). The construction follows the Lyapunov--feedback
iteration used in the finite-dimensional bounded-real proof of
\cite[Lemma~3.6]{wang2023stochastic}, with the operator-norm convergence argument
adapted from the infinite-dimensional Riccati iteration in
\cite[Theorem~8.2]{wang2025lqjumps}.

\smallskip
\noindent\textbf{Step 1: uniform convexity and the initial Lyapunov solution.}
Choose $\varepsilon>0$ such that
\[
0<\varepsilon^2\le\gamma^2-\|\mathcal L\|^2.
\]
Then, for every $v\in\mathcal V[t,T]$,
\begin{equation}\label{eq:gain_uniform_convexity}
\begin{aligned}
J_1(t,0;0,v)
&=\gamma^2\|v\|_{\mathcal V[t,T]}^2
-\|\mathcal Lv\|_{L^2_{\mathbb F}(t,T;Z_1)}^2\\
&\ge\varepsilon^2\|v\|_{\mathcal V[t,T]}^2.
\end{aligned}
\end{equation}
Let $P_0\in C_{\mathcal S}([t,T];\mathcal S(H))$ be the unique mild solution of
\begin{equation}\label{eq:iteration_p0}
\left\{
\begin{aligned}
&0={}\dot P_0+P_0(A+A_1)+(A+A_1)^*P_0\\
&\quad+A_2^*P_0A_2-M^*M,\\
&P_0(T)=0.
\end{aligned}
\right.
\end{equation}
Existence, uniqueness, and the operator estimate follow from
Lemma~\ref{lem:mild_lyapunov_uniform_estimate}; this is also the no-jump
specialization of the mild Lyapunov construction in
\cite[Definition~8.1 and Theorem~8.2, Step~1]{wang2025lqjumps}.
Lemma~\ref{lem:j1_lower_bound} yields a constant
$\mu>0$ such that
\begin{equation}\label{eq:uniform_cost_lower_bound}
J_1(\tau,\xi;0,v)\ge-\mu\mathbb E\|\xi\|_H^2
\end{equation}
for every $\tau\in[t,T]$ and every admissible pair $(\xi,v)$ on $[\tau,T]$.
Moreover, Lemma~\ref{lem:lambda_coercivity}, applied with $\varphi=0$, gives
\begin{equation}\label{eq:iteration_coercivity_p0}
\Lambda^\gamma(P_0(s))\ge\varepsilon^2I_V
\quad\text{for a.e. }s\in[t,T].
\end{equation}

\smallskip
\noindent\textbf{Step 2: Lyapunov--feedback iteration and monotonicity.}
Suppose $P_j$ has been constructed and satisfies
$\Lambda^\gamma(P_j)\ge\varepsilon^2I_V$. Define
\begin{equation}\label{eq:iteration_feedback}
\begin{aligned}
\varphi_j&=-\bigl(\Lambda^\gamma(P_j)\bigr)^{-1}\Phi(P_j)^*,\\
\widehat A_j&=A_1+C_1\varphi_j,\\
\widehat C_j&=A_2+C_2\varphi_j.
\end{aligned}
\end{equation}
and let $P_{j+1}$ be the unique mild solution of
\begin{equation}\label{eq:iteration_lyapunov}
\left\{
\begin{aligned}
&0={}\dot P_{j+1}+P_{j+1}(A+\widehat A_j)+(A+\widehat A_j)^*P_{j+1}\\
&+\widehat C_j^*P_{j+1}\widehat C_j+\gamma^2\varphi_j^*\varphi_j-M^*M,\\
&P_{j+1}(T)=0,
\end{aligned}
\right.
\end{equation}
where the equation is understood in the mild sense and has a unique solution by
Lemma~\ref{lem:mild_lyapunov_uniform_estimate}. Taking $v=0$ in
Lemma~\ref{lem:closed_loop_perturbation} and then using
\eqref{eq:uniform_cost_lower_bound}, we obtain, for deterministic $x\in H$,
\[
\langle P_{j+1}(s)x,x\rangle
=J_1(s,x;0,\varphi_jX)\ge-\mu\|x\|_H^2.
\]
Hence $P_{j+1}(s)\ge-\mu I_H$. Lemma~\ref{lem:lambda_coercivity}, now applied to
the pair $(P_{j+1},\varphi_j)$, also gives
\begin{equation}\label{eq:iteration_uniform_coercivity}
\Lambda^\gamma(P_{j+1}(s))\ge\varepsilon^2I_V.
\end{equation}
Thus the iteration is well defined for all $j\ge0$.

Put $\varphi_{-1}=0$ and, for $j\ge0$, set
\[
\Delta_j=P_j-P_{j+1},
\qquad
\Theta_j=\varphi_{j-1}-\varphi_j.
\]
Subtracting the two consecutive Lyapunov equations and using the stationarity
identity
\[
\Phi(P_j)^*+\Lambda^\gamma(P_j)\varphi_j=0
\]
shows that $\Delta_j$ is the mild solution of
\begin{equation}\label{eq:iteration_difference}
\left\{
\begin{aligned}
&0={}\dot\Delta_j+\Delta_j(A+\widehat A_j)+(A+\widehat A_j)^*\Delta_j\\
&+\widehat C_j^*\Delta_j\widehat C_j+\Theta_j^*\Lambda^\gamma(P_j)\Theta_j,\\
&\Delta_j(T)=0.
\end{aligned}
\right.
\end{equation}
The forcing operator in \eqref{eq:iteration_difference} is nonnegative.
The representation \eqref{eq:positive_lyapunov_representation} therefore yields
$\Delta_j(s)\ge0$. Consequently, for $s\in[t,T]$,
\begin{equation}\label{eq:iteration_monotone}
\begin{aligned}
P_0(s)\ge P_1(s)\ge\cdots\ge P_j(s)\ge P_{j+1}(s)\ge-\mu I_H,
\end{aligned}
\end{equation}
Since every $P_j(s)$ is self-adjoint, the order bounds give the explicit estimate
\begin{equation}\label{eq:iteration_uniform_operator_bound}
\begin{aligned}
&\sup_{j\ge0}\sup_{s\in[t,T]}\|P_j(s)\|_{\mathcal L(H)}\\
&\quad\le \max\bigl\{\mu,
\|P_0\|_{L^\infty(t,T;\mathcal L(H))}\bigr\}<\infty.
\end{aligned}
\end{equation}

\smallskip
\noindent\textbf{Step 3: operator-norm convergence and passage to the Riccati
equation.}
Write
\[
K_j=\Lambda^\gamma(P_j),
\qquad L_j=\Phi(P_j)^*.
\]
For $j\ge1$, the resolvent identity for $K_j^{-1}$ gives
\begin{equation}\label{eq:iteration_feedback_difference}
\begin{aligned}
&\Theta_j=K_j^{-1}(K_{j-1}-K_j)K_{j-1}^{-1}L_j-K_{j-1}^{-1}(L_{j-1}-L_j),\\
&K_{j-1}-K_j=C_2^*\Delta_{j-1}C_2,\\
&L_{j-1}-L_j=C_1^*\Delta_{j-1}+C_2^*\Delta_{j-1}A_2.
\end{aligned}
\end{equation}
By \eqref{eq:iteration_uniform_coercivity},
$\|K_j^{-1}\|\le\varepsilon^{-2}$. Together with
\eqref{eq:iteration_monotone} and the boundedness of the coefficients,
\eqref{eq:iteration_feedback_difference} implies, with a constant independent of
$j$,
\begin{equation}\label{eq:theta_delta_bound}
\begin{aligned}
\|\Theta_j(s)\|&\le C\|\Delta_{j-1}(s)\|,\\
\|\Theta_j^*K_j\Theta_j\|
&\le C\|\Delta_{j-1}(s)\|.
\end{aligned}
\end{equation}
The coercivity of $K_j$, the uniform bound for $P_j$, and
\eqref{eq:iteration_feedback} also show that $\{\varphi_j\}$, and hence
$\{\widehat A_j\}$ and $\{\widehat C_j\}$, are uniformly essentially bounded.
The second estimate in \eqref{eq:theta_delta_bound} bounds the forcing operator in
\eqref{eq:iteration_difference}. Since the pairs
$(\widehat A_j,\widehat C_j)$ are bounded uniformly in $j$, the uniform-family
part of Lemma~\ref{lem:mild_lyapunov_uniform_estimate} gives
\begin{equation}\label{eq:iteration_gronwall}
\|\Delta_j(s)\|
\le C_0\int_s^T\|\Delta_{j-1}(r)\|\,dr,
\qquad j\ge1,
\end{equation}
where $C_0$ is independent of $j$. If
$D_0=\sup_{s\in[t,T]}\|\Delta_0(s)\|$, induction yields
\begin{equation}\label{eq:iteration_factorial}
\|\Delta_j(s)\|
\le D_0\frac{[C_0(T-s)]^j}{j!},
\qquad s\in[t,T].
\end{equation}
Thus $\sum_{j\ge0}\Delta_j$ converges uniformly in operator norm, so that
$P_j\to P$ uniformly on $[t,T]$ for some
$P\in C_{\mathcal S}([t,T];\mathcal S(H))$. This factorial estimate replaces the
finite-dimensional compactness/Ascoli step in \cite{wang2023stochastic} and is the
essential infinite-dimensional ingredient from \cite{wang2025lqjumps}.
Notice that no operator-norm continuity of the individual $P_j$ is required:
statewise continuity in $C_{\mathcal S}$, together with uniform convergence in
operator norm, is sufficient to place the limit $P$ in $C_{\mathcal S}$.

The uniform convergence implies
\[
\begin{aligned}
&K_j\to\Lambda^\gamma(P),\\
&L_j\to\Phi(P)^*,\\
&\varphi_j\to\varphi
:=-\bigl(\Lambda^\gamma(P)\bigr)^{-1}\Phi(P)^*.
\end{aligned}
\]
in the corresponding essential-supremum operator norms. Moreover,
\eqref{eq:iteration_uniform_coercivity} passes to the limit and gives
\begin{equation}\label{eq:limit_strong_regularity}
\Lambda^\gamma(P(s))\ge\varepsilon^2I_V
\quad\text{for a.e. }s\in[t,T].
\end{equation}
Passing to the limit in the mild form of \eqref{eq:iteration_lyapunov} gives
\[
\begin{aligned}
&0={}\dot P+P(A+A_1+C_1\varphi)+(A+A_1+C_1\varphi)^*P\\
&+(A_2+C_2\varphi)^*P(A_2+C_2\varphi)+\gamma^2\varphi^*\varphi-M^*M,\\
&P(T)=0,
\end{aligned}
\]
in the mild sense. Substituting
$\varphi=-(\Lambda^\gamma(P))^{-1}\Phi(P)^*$ reduces this equation exactly to
\eqref{eq:hinf_riccati}. Therefore $P$ is a strongly regular mild solution.

\smallskip
\noindent\textbf{Step 4: uniqueness, feedback, and value.}
Let $P$ and $\widetilde P$ be two strongly regular mild solutions. On the bounded
set containing their ranges, with the two coercivity constants replaced by their
minimum, the nonlinear map
\[
\mathcal N(R)=\Phi(R)\big(\Lambda^\gamma(R)\big)^{-1}\Phi(R)^*
\]
is locally Lipschitz in operator norm. Indeed, this follows from
$K^{-1}-\widetilde K^{-1}=K^{-1}(\widetilde K-K)\widetilde K^{-1}$ and the linear
dependence of $\Phi(R)$ and $\Lambda^\gamma(R)$ on $R$. Subtracting the two mild
Riccati identities therefore yields
\[
\|P(s)-\widetilde P(s)\|
\le C\int_s^T\|P(r)-\widetilde P(r)\|\,dr.
\]
Backward Gronwall's inequality gives $P=\widetilde P$.

Finally, take $\varphi=\Psi_P$ in
Lemma~\ref{lem:closed_loop_perturbation}. Since
$\mathfrak N_{\Psi_P}=0$ and the corresponding input transformation is bijective
by Proposition~\ref{prop:feedback_input_isomorphism},
we may rename the original disturbance. For a nonzero initial state this
transformation is the same bounded isomorphism plus the fixed shift
$-\Psi_PX^{t,\xi;0}$, and hence remains bijective. We obtain, for every admissible $v$,
\[
\begin{aligned}
J_1(t,\xi;0,v)
&=\mathbb E\langle\xi,P(t)\xi\rangle\\
&\hspace{-0.3cm}\quad+\mathbb E\int_t^T
\langle\Lambda^\gamma(P)(v-\Psi_PX),v-\Psi_PX\rangle\,ds.
\end{aligned}
\]
The last term is nonnegative and vanishes exactly for the unique closed-loop
solution $v^*=\Psi_PX^*$. This proves the feedback formula and the value identity.
\end{pf}

\section{Closed-Loop Nash Verification via Coupled Operator Riccati Equations}

This section investigates the closed-loop Nash equilibrium problem for the
infinite-dimensional mixed $H_2/H_\infty$ control system. The two performance criteria
lead to a system of coupled operator Riccati equations.

To characterize the coupled game strategies, let the feedback gains satisfy
\[
\begin{aligned}
K_U(\cdot)&\in L^\infty(t,T;\mathcal L(H;U)),\\
K_V(\cdot)&\in L^\infty(t,T;\mathcal L(H;V)).
\end{aligned}
\]
We use the following notation:
\begin{align*}
&\mathcal A_U=A+A_1+B_1K_U,\mathcal C_U=A_2+B_2K_U,\\
&\mathcal A_V=A+A_1+C_1K_V,
\mathcal C_V=A_2+C_2K_V,\\
&\Lambda_V(P_1) = \gamma^2 I_V + C_2^* P_1 C_2, \\
&\Lambda_U(P_2) = I_U + B_2^* P_2 B_2, \\
&\Upsilon_V(P_1,K_U)=P_1C_1+\mathcal C_U^*P_1C_2, \\
&\Upsilon_U(P_2,K_V)=P_2B_1+\mathcal C_V^*P_2B_2.
\end{align*}

\begin{assum} \textit{(Cross-coupled Riccati solvability)} \label{assum:cross_riccati}
There exist $P_1,P_2\in C_{\mathcal S}([t,T];\mathcal S(H))$ and feedback gains $K_U\in L^\infty(t,T;\mathcal L(H;U))$, $K_V\in L^\infty(t,T;\mathcal L(H;V))$ such that $P_1(T)=P_2(T)=0$ and the following cross-coupled equations hold in the mild form specified below:
\begin{equation}
\left\{
\begin{aligned}
&0={}\dot P_1+\mathcal A_U^*P_1+P_1\mathcal A_U
+\mathcal C_U^*P_1\mathcal C_U-M^*M\\
&-K_U^*K_U-\Upsilon_V(P_1,K_U)\Lambda_V(P_1)^{-1}
\Upsilon_V(P_1,K_U)^*,\\
&\Lambda_V(P_1(s))\ge \delta_1 I_V,
\end{aligned}
\right.
\label{eq:riccati_v}
\end{equation}
\noindent and
\begin{equation}
\left\{
\begin{aligned}
&0={}\dot P_2+\mathcal A_V^*P_2+P_2\mathcal A_V
+\mathcal C_V^*P_2\mathcal C_V+M^*M\\
&-\Upsilon_U(P_2,K_V)\Lambda_U(P_2)^{-1}
\Upsilon_U(P_2,K_V)^*,\\
&\Lambda_U(P_2(s))\ge \delta_2 I_U,
\end{aligned}
\right.
\label{eq:riccati_u}
\end{equation}
for some constants $\delta_1,\delta_2>0$. Moreover,
\begin{equation}
\begin{aligned}
K_V&=-\Lambda_V(P_1)^{-1}\Upsilon_V(P_1,K_U)^*,\\
K_U&=-\Lambda_U(P_2)^{-1}\Upsilon_U(P_2,K_V)^*.
\end{aligned}
\label{eq:feedback_gains}
\end{equation}
We refer to \eqref{eq:riccati_v}--\eqref{eq:feedback_gains} collectively as the
cross-coupled Riccati equations.
Put $F_U=A_1+B_1K_U$, $G_U=A_2+B_2K_U$,
$F_V=A_1+C_1K_V$, and $G_V=A_2+C_2K_V$. Define
\[
\begin{aligned}
\mathfrak R_1&=P_1F_U+F_U^*P_1+G_U^*P_1G_U-M^*M-K_U^*K_U\\
&\quad-\Upsilon_V\Lambda_V^{-1}\Upsilon_V^*,\\
\mathfrak R_2&=P_2F_V+F_V^*P_2+G_V^*P_2G_V+M^*M\\
&\quad-\Upsilon_U\Lambda_U^{-1}\Upsilon_U^*.
\end{aligned}
\]
The precise meaning of the two equations is
\begin{equation}\label{eq:coupled_mild}
\begin{aligned}
P_i(s)x&=\int_s^T S(r-s)^*\mathfrak R_i(r)S(r-s)x\,dr,\\
&\hspace{2cm}x\in H,\quad i=1,2.
\end{aligned}
\end{equation}
Their differential displays are shorthand for these identities. Testing on
$D(A)$ recovers the weak equations. The quadratic identities are justified
by Lemma~\ref{lem:general_mild_identity}; no membership in the optional
$W^{1,1}$ subclass is assumed.
\end{assum}

Before giving the Nash verification theorem, we state the $H_2$ quadratic identity that is
the analogue of Lemma~\ref{lem:j1_expansion} for the cost functional $J_2$, specialized to the
system with a fixed disturbance feedback $v=K_VX$.

\begin{lem} \textit{($H_2$ quadratic identity)} \label{lem:j2_expansion}
Let Assumptions~\ref{assum:wellposed} and~\ref{assum:output} hold. Fix a disturbance feedback
$K_V\in L^\infty(t,T;\mathcal L(H;V))$ and let $P_2\in C_{\mathcal S}([t,T];\mathcal S(H))$ satisfy
$P_2(T)=0$ and the $H_2$ Riccati identity in Assumption~\ref{assum:cross_riccati}
in the mild sense. Then for any initial state
$\xi\in L^2(\Omega,\mathcal F_t,\mathbb P;H)$ and any control $u\in\mathcal U[t,T]$, the
mild solution $X$ of the system below, where
$A_{K_V}=A+A_1+C_1K_V$ and $C_{K_V}=A_2+C_2K_V$,
\[
\begin{aligned}
&dX=\big[A_{K_V}X+B_1u\big]ds
+\big[C_{K_V}X+B_2u\big]dW,\\
&X(t)=\xi,
\end{aligned}
\]
satisfies
\[
\begin{aligned}
J_2(t,\xi;u,K_VX)
&=\mathbb E\langle \xi,P_2(t)\xi\rangle\\
&\hspace{-1.0cm}\quad+\mathbb E\int_t^T
\Big\langle\begin{pmatrix}X\\u\end{pmatrix},Q_2(P_2,K_V)
\begin{pmatrix}X\\u\end{pmatrix}\Big\rangle\,ds,
\end{aligned}
\]
where
\[
Q_2(P_2,K_V)=
\begin{pmatrix}
\mathcal S_2(P_2,K_V)&\Upsilon_U(P_2,K_V)\\
\Upsilon_U(P_2,K_V)^*&\Lambda_U(P_2)
\end{pmatrix}
\]
with
\[
\begin{aligned}
\mathcal S_2(P_2,K_V)
&=\dot P_2+A_{K_V}^*P_2+P_2A_{K_V}\\
&\quad+C_{K_V}^*P_2C_{K_V}+M^*M .
\end{aligned}
\]
The block differential notation is evaluated by substituting the Riccati equation; the resulting integrand is bounded. The identity follows in the mild sense of Lemma~\ref{lem:general_mild_identity}.
\end{lem}

\begin{pf}
Apply Lemma~\ref{lem:general_mild_identity} with
$F=A_1+C_1K_V$, $G=C_{K_V}$, $f=B_1u$, $g=B_2u$ and
\[
Q=M^*M-\Upsilon_U\Lambda_U^{-1}\Upsilon_U^*.
\]
Adding $\mathbb E\int_t^T\|u\|^2ds$ and collecting the cross and
quadratic terms gives the stated block expression. Substitution of the
Riccati identity makes its upper-left block
$\Upsilon_U\Lambda_U^{-1}\Upsilon_U^*$, so every term is defined
on $H\times U$ without applying $A$ to the mild state.
\end{pf}

\begin{thm} \textit{(Conditional Nash verification)} \label{thm:sufficiency_nash}
Under Assumptions~\ref{assum:wellposed}, \ref{assum:output}, and~\ref{assum:cross_riccati}, for any initial pair $(t,\xi)$, let $X^*$ be the mild solution to the closed-loop system:
\[
\begin{cases}
dX^*(s)=\big[A+A_1+B_1K_U+C_1K_V\big]X^*(s)\,ds \\
\qquad+\big[A_2+B_2K_U+C_2K_V\big]X^*(s)\,dW(s),\\
X^*(t)=\xi.
\end{cases}
\]
Then, the strategies $u^*(s)=K_U(s)X^*(s)$ and $v^*(s)=K_V(s)X^*(s)$ constitute a closed-loop Nash equilibrium satisfying:
\[
\begin{aligned}
J_1(t,\xi;u^*,v^*) \le J_1(t,\xi;K_UX,v), \\
J_2(t,\xi;u^*,v^*) \le J_2(t,\xi;u,K_VX).
\end{aligned}
\]
Furthermore, the fixed-control disturbance operator, whose output is the full
performance variable $z=(MX,K_UX)$, satisfies the strict bounded-real condition
$\|\mathcal L_{K_U}\|<\gamma$.
\end{thm}

\begin{pf}
The ensuing quadratic identities follow from Lemma~\ref{lem:general_mild_identity}, applied to the two fixed-feedback subsystems.

Fixing $u=K_UX$ for an arbitrary $v\in\mathcal V[t,T]$, the state satisfies:
\[
\begin{aligned}
&dX = \big[\mathcal A_UX+C_1v\big]ds+\big[\mathcal C_UX+C_2v\big]dW, \\
&X(t) = \xi.
\end{aligned}
\]
Invoking the $H_\infty$ mild quadratic identity and \eqref{eq:riccati_v}, completing the square with $K_V=-\Lambda_V(P_1)^{-1}\Upsilon_V(P_1,K_U)^*$ yields:
\[
\begin{aligned}
J_1(t,\xi;K_UX,v) &= \mathbb E\langle \xi,P_1(t)\xi\rangle \\
&\hspace{-0.3cm}\quad + \mathbb E\int_t^T \big\|\Lambda_V(P_1)^{1/2}(v-K_VX)\big\|_V^2\,ds.
\end{aligned}
\]
The coercivity $\Lambda_V(P_1)\ge\delta_1I_V$ shows that equality is possible
only when $v=K_VX$. The bounded-feedback closed-loop equation is well posed, so
this relation determines the unique minimizing response $v^*=K_VX^*$.

Symmetrically, fixing $v=K_VX$ for an arbitrary $u\in\mathcal U[t,T]$, Lemma~\ref{lem:j2_expansion} and \eqref{eq:riccati_u} allow an analogous completion of squares via $K_U=-\Lambda_U(P_2)^{-1}\Upsilon_U(P_2,K_V)^*$ to produce:
\[
\begin{aligned}
J_2(t,\xi;u,K_VX) &= \mathbb E\langle \xi,P_2(t)\xi\rangle \\
&\hspace{-0.3cm}\quad + \mathbb E\int_t^T \big\|\Lambda_U(P_2)^{1/2}(u-K_UX)\big\|_U^2\,ds.
\end{aligned}
\]
Since $\Lambda_U(P_2)\ge\delta_2I_U$, equality is possible only when
$u=K_UX$. Well-posedness of the corresponding bounded-feedback equation then
determines $u^*=K_UX^*$ uniquely and establishes the closed-loop Nash
equilibrium.

It remains to verify that the attenuation conclusion concerns the full output
used in Section~\ref{sec:game}. Set $\xi=0$ in the first quadratic identity. Because
$u=K_UX$, its left-hand side is exactly
\[
\gamma^2\|v\|_{\mathcal V[t,T]}^2
-\|\mathcal L_{K_U}v\|_{L^2_{\mathbb F}(t,T;Z)}^2,
\]
where
$\|\mathcal L_{K_U}v\|^2=\mathbb E\int_t^T
(\|MX\|_{Z_1}^2+\|K_UX\|_U^2)\,ds$.
For the fixed-control subsystem, Proposition~\ref{prop:feedback_input_isomorphism}
applies with the bounded coefficients $A_1+B_1K_U$ and $A_2+B_2K_U$ and with
$\Psi=K_V$. Hence there is an $m_U>0$ such that
\[
\|v-K_VX\|_{\mathcal V[t,T]}\ge m_U\|v\|_{\mathcal V[t,T]}.
\]
Together with $\Lambda_V(P_1)\ge\delta_1I_V$, the quadratic identity yields
\[
\gamma^2\|v\|_{\mathcal V[t,T]}^2
-\|\mathcal L_{K_U}v\|_{L^2_{\mathbb F}(t,T;Z)}^2
\ge\delta_1m_U^2\|v\|_{\mathcal V[t,T]}^2.
\]
Taking the supremum over nonzero $v$ gives
$\|\mathcal L_{K_U}\|^2\le\gamma^2-\delta_1m_U^2<\gamma^2$, and therefore
$\|\mathcal L_{K_U}\|<\gamma$.
\end{pf}

\begin{rem}
For a verified pair, $P_2(s)\ge0$, because its value is the nonnegative
cost $J_2$. Also $P_1(s)\le0$, since the admissible response $v=0$
has nonpositive $J_1$. In particular $\Lambda_U(P_2)\ge I_U$.
These observations do not establish existence of a coupled pair on a
prescribed horizon. General regular Riccati characterizations may involve
generalized inverses and range conditions; ordinary bounded inverse
feedback formulas require the strong coercivity stated here.
\end{rem}

\subsection{Local Solvability of the Coupled Mild Equations}

\begin{thm}\textit{(Local terminal-interval existence and uniqueness)}
\label{thm:local_coupled}
Under Assumptions~\ref{assum:wellposed} and~\ref{assum:output}, fix
$\gamma>0$. There exist $h>0$ with $h\le T$ and $\rho>0$ such that
\eqref{eq:coupled_mild} and \eqref{eq:feedback_gains} have a unique
solution in the ball
\[
\max_{i=1,2}\sup_{s\in[T-h,T]}\|P_i(s)\|\le\rho
\]
of the mild solution space. The feedbacks are bounded and measurable, and
\[
\Lambda_U\ge\tfrac12 I_U,\qquad
\Lambda_V\ge\tfrac{\gamma^2}{2}I_V.
\]
The resulting pair is a closed-loop Nash strategy with full-output strict
attenuation on $[T-h,T]$. Uniqueness is asserted in this ball; extension
to the whole prescribed horizon is not asserted.
\end{thm}
\begin{pf}
Expand the two stationarity identities before solving for the feedbacks.
On $U\times V$ they form the block system
\begin{equation}\label{eq:feedback_block}
\mathbb B(P_1,P_2)\begin{pmatrix}K_U\\K_V\end{pmatrix}
=-\begin{pmatrix}B_1^*P_2+B_2^*P_2A_2\\
C_1^*P_1+C_2^*P_1A_2\end{pmatrix},
\end{equation}
where
\[
\mathbb B(P_1,P_2)=
\begin{pmatrix}
I_U+B_2^*P_2B_2&B_2^*P_2C_2\\
C_2^*P_1B_2&\gamma^2I_V+C_2^*P_1C_2
\end{pmatrix}.
\]
Let $\mathbb B_0=\operatorname{diag}(I_U,\gamma^2I_V)$.
Choose $\rho>0$ sufficiently small, using the essential-supremum
coefficient bounds, that for every pair with $\|P_i\|\le\rho$,
\[
\begin{aligned}
\|\mathbb B_0^{-1}(\mathbb B-\mathbb B_0)\|&\le\tfrac12,\\
\rho\|B_2\|_\infty^2&\le\tfrac12,\qquad
\rho\|C_2\|_\infty^2\le\tfrac{\gamma^2}{2}.
\end{aligned}
\]
The Neumann series gives a uniformly bounded inverse of $\mathbb B$.
Equation~\eqref{eq:feedback_block} therefore defines feedback maps
$\mathcal K_U(P_1,P_2)$ and $\mathcal K_V(P_1,P_2)$. They vanish at
$(0,0)$ and are uniformly bounded and Lipschitz in operator norm on
this ball. Indeed, the right-hand side is linear in the pair and
$\mathbb B^{-1}-\widetilde{\mathbb B}^{-1}
=\mathbb B^{-1}(\widetilde{\mathbb B}-\mathbb B)
\widetilde{\mathbb B}^{-1}$. The same Neumann series shows strong
operator measurability of the feedbacks and of their adjoints.
The two weight lower bounds follow from the last two inequalities.

Insert these maps into $\mathfrak R_i$. Products, adjoints and the
uniformly coercive weight inverses give constants $C_\rho,L_\rho<\infty$
such that, a.e. in time,
\[
\begin{aligned}
\max_i\|\mathfrak R_i(P)\|&\le C_\rho,\\
\max_i\|\mathfrak R_i(P)-\mathfrak R_i(\widetilde P)\|
&\le L_\rho\max_i\|P_i-\widetilde P_i\|.
\end{aligned}
\]
These constants depend only on $\gamma,\rho$ and the coefficient bounds.
Let $M_S=\sup_{0\le r\le T}\|S(r)\|$. On the complete closed ball
of pairs in $C_{\mathcal S}([T-h,T];\mathcal S(H))$, define
$\mathcal T$ by the right-hand side of \eqref{eq:coupled_mild}.
Its forcing is self-adjoint and bounded, so $\mathcal T$ takes values
in this space: statewise continuity follows from strong continuity of
$S,S^*$ and dominated convergence in the moving-limit integrals.
Moreover,
\[
\begin{aligned}
\|\mathcal T P\|_\infty&\le M_S^2hC_\rho,\\
\|\mathcal T P-\mathcal T\widetilde P\|_\infty
&\le M_S^2hL_\rho\|P-\widetilde P\|_\infty.
\end{aligned}
\]
Choose $0<h\le T$ with $M_S^2hC_\rho\le\rho$ and
$M_S^2hL_\rho<1$. The contraction theorem gives the unique fixed
point in the ball and its bounded feedbacks. It satisfies the mild
equations, stationarity and strong regularity, hence
Theorem~\ref{thm:sufficiency_nash} applies. This argument uses bounded
perturbations of an arbitrary $C_0$ generator and does not require
operator-norm time regularity of the Riccati pair.
\end{pf}

\begin{rem}
Local existence holds for each fixed $\gamma>0$ because the zero terminal
weights make the feedback block invertible. The interval supplied by the
proof may shrink as $\gamma$ decreases. This local theorem does not identify the numerical branch endpoint. The selected heat-equation example is established on the full horizon by the separate invariant-region argument in Theorem~\ref{thm:heat_global}. Beyond the small ball, positivity of the two
diagonal weights alone is not a substitute for invertibility of the
coupled feedback block.
\end{rem}

\section{An Illustrative Galerkin Projection for a Stochastic Heat Equation}

This section projects the abstract framework onto finite-dimensional spaces, verifying semigroup generation, coefficient boundedness, and output normalization for a distributed parameter model. Theorem~\ref{thm:local_coupled} guarantees a mild Nash design on a sufficiently short terminal interval for this model. For the selected heat-equation coefficients at $\gamma=0.09$, Theorem~\ref{thm:heat_global} additionally proves full-horizon solvability and spectral convergence analytically.

\subsection{Stochastic Heat-Equation Model}
\label{subsec:heat_model}

Let $H=L^2(0,1)$ and let $A$ be the Dirichlet Laplacian,
\[
Ax=x'',\qquad D(A)=H^2(0,1)\cap H_0^1(0,1).
\]
The operator $A$ is self-adjoint and negative and generates an analytic contraction
semigroup on $H$. Consider the stochastic heat equation
\[
\begin{aligned}
&dX(s)=\big[(A+A_1)X(s)+B_1u(s)+C_1v(s)\big]\,ds\\
&\quad+\big[A_2X(s)+B_2u(s)+C_2v(s)\big]\,dW(s),\\
&z(s)=\begin{pmatrix}MX(s)\\ Nu(s)\end{pmatrix},
\qquad X(t)=\xi .
\end{aligned}
\]
For a concrete illustration, take $U=V=Z_1=Z_2=H$ and
\[
\begin{aligned}
&A_1=0.1I,\quad A_2=0.2I,\quad B_1=I,\\
&B_2=0.5I,\quad C_1=C_2=0.3I,\\
&M=I,\quad N=I.
\end{aligned}
\]
With bounded, deterministic coefficients excluding $A$, and $N^*N=I_U$, this state-control-disturbance dependent diffusion model satisfies Assumptions~\ref{assum:wellposed} and~\ref{assum:output}. Physically, the scalar parameters represent a weak reaction term $A_1$, noise intensities $A_2$ and $B_2$, distributed disturbance couplings $C_1$ and $C_2$, and unit observation weights $M$ and $N$. As identity multiples, these coefficients ensure modal decoupling. This avoids spatial approximations within retained modes, requiring only numerical time integration for the scalar Riccati equations. Modal decoupling alone does not guarantee Riccati solvability; the invariant-region argument below supplies the necessary uniform bounds at $\gamma=0.09$.
\subsection{A Full-Horizon Solvable Heat-Equation Case}

\begin{thm}\textit{(Uniform modal bounds and a full-horizon Nash design)}
\label{thm:heat_global}
For the heat-equation coefficients above and $\gamma=0.09$, the coupled
mild Riccati system has a diagonal strongly regular solution on every
finite horizon $[0,T]$. Writing its modal values as $p_{i,k}(T-s)$,
the following bounds hold for every $k$ and $s$:
\[
-0.062\le p_{1,k}\le0,\qquad 0\le p_{2,k}\le0.12,
\]
\[
\begin{aligned}
\Lambda_{V,k}&\ge0.00252,\qquad \Lambda_{U,k}\ge1,\\
|K_{U,k}|&<0.266,\qquad |K_{V,k}|\le\frac{62}{7}.
\end{aligned}
\]
The diagonal feedbacks satisfy the full-horizon Nash and strict-attenuation
conclusions of Theorem~\ref{thm:sufficiency_nash}. Moreover, the spectral
Riccati and feedback operators, extended by zero off $H_n$, converge
uniformly in time in operator norm, with errors $O((n+1)^{-2})$.
For every fixed initial datum, the projected equilibrium states and
inputs converge in $C([0,T];L^2(\Omega;H))$ and $L^2_{\mathbb F}$,
respectively, and their two game costs converge.
\end{thm}
\begin{pf}
Set $a_k=0.1-k^2\pi^2$, $g=0.09$, and $\tau=T-s$. The scalar
stationarity system has determinant
\[
d=g^2(1+0.25p_2)+0.09p_1
\]
and gives
\begin{equation}\label{eq:heat_explicit_feedback}
\begin{aligned}
k_u&=-\frac{p_2(1.1g^2+0.045p_1)}{d},\\
k_v&=\frac{p_1(-0.36+0.075p_2)}{d}.
\end{aligned}
\end{equation}
With $c_u=0.2+0.5k_u$, $c_v=0.2+0.3k_v$, the modal ODEs are
\begin{equation}\label{eq:heat_modal_odes}
\begin{aligned}
\frac{dp_1}{d\tau}&=[2(a_k+k_u)+c_u^2]p_1-1-k_u^2\\
&\quad-\frac{[p_1(0.3+0.3c_u)]^2}{g^2+0.09p_1},\\
\frac{dp_2}{d\tau}&=[2(a_k+0.3k_v)+c_v^2]p_2+1\\
&\quad-\frac{[p_2(1+0.5c_v)]^2}{1+0.25p_2}.
\end{aligned}
\end{equation}
with $p_1(0)=p_2(0)=0$. Their vector field is smooth on a neighborhood
of the rectangle $\mathcal D=[-r,0]\times[0,b]$, where
$r=31/500$ and $b=3/25$. Indeed,
$g^2-0.09r=63/25000$ and $d\ge63/25000$ on this rectangle.

We verify inward pointing on all four faces. At $p_1=0$ the first
derivative is $-1-k_u^2<0$, and at $p_2=0$ the second derivative is $1$.
On $p_1=-r$, formula~\eqref{eq:heat_explicit_feedback} gives
$u_b\le k_u\le0$, where
\[
u_b=-\frac{b(1.1g^2-0.045r)}{g^2(1+0.25b)-0.09r}
=-\frac{408}{1535}.
\]
The first derivative on this face is the concave quadratic in $u=k_u$
\[
\begin{aligned}
f_1(u)&=-r[2(a_k+u)+(0.2+0.5u)^2]-1-u^2\\
&\quad-\frac{r^2(0.36+0.15u)^2}{g^2-0.09r}.
\end{aligned}
\]
Consequently its minimum over $[u_b,0]$ occurs at an endpoint.
Since $a_k\le-9.75$, substitution of the rational upper bound
$a_k=-9.75$ yields $f_1(0)>0.008$ and $f_1(u_b)>0.014$.
Decreasing $a_k$ only increases $f_1$, so this face points inward for
every mode.

On $p_2=b$, the disturbance gain lies in $[0,v_b]$, where
\[
v_b=\frac{r(0.36-0.075b)}{g^2(1+0.25b)-0.09r}
=\frac{2418}{307}.
\]
The second derivative on this face, expressed as a quadratic in $v=k_v$,
is
\[
\begin{aligned}
f_2(v)&=b[2(a_k+0.3v)+(0.2+0.3v)^2]+1\\
&\quad-\frac{b^2(1.1+0.15v)^2}{1+0.25b}.
\end{aligned}
\]
Its quadratic coefficient is $0.09b/(1+0.25b)>0$, and its linear
coefficient is $0.72b-0.33b^2/(1+0.25b)>0$.
It is therefore increasing for $v\ge0$. At $a_k=-9.75$,
direct rational substitution gives $f_2(v_b)<-0.057$; a smaller
$a_k$ decreases this derivative. This proves invariance of $\mathcal D$.
The endpoint inequalities involve only rational arithmetic; the use of
$a_k\le-9.75$ is justified already by $\pi>3.14$.

The scalar ODE solution starting at the origin consequently stays in
$\mathcal D$ and extends to every finite $\tau$. The control magnitude
in \eqref{eq:heat_explicit_feedback} increases with $-p_1$ and $p_2$
on this rectangle, hence it is at most $|u_b|<0.266$.
The disturbance magnitude increases with $-p_1$ and decreases with
$p_2$, hence it is at most $0.36r/(g^2-0.09r)=62/7$.
Define $P_i(s)e_k=p_{i,k}(T-s)e_k$ and define $K_U,K_V$ similarly.
The uniform bounds and coordinatewise continuity give bounded,
self-adjoint strongly continuous $P_i$ and bounded measurable feedbacks.
The scalar variation-of-constants identities give
\eqref{eq:coupled_mild} first on finite modal sums and then on $H$ by
dominated convergence. The weight bounds thus verify every hypothesis
of Theorem~\ref{thm:sufficiency_nash} on the full horizon.

Finally, write each modal ODE as $p'_{i,k}=2a_kp_{i,k}+b_{i,k}$.
All terms in $b_{i,k}$ are uniformly bounded by the rectangle, feedback
and weight bounds, with a constant $C$ independent of $k$ and time.
Variation of constants gives
\[
|p_{i,k}(\tau)|\le\frac{C}{2(k^2\pi^2-0.1)}.
\]
The explicit feedback formulas and $d\ge0.00252$ then give the same
$O(k^{-2})$ decay for $|K_{U,k}|+|K_{V,k}|$. Since a diagonal
operator norm is the supremum of its modal magnitudes, these estimates
prove the asserted operator-norm convergence of the zero-extended
spectral operators. The projected closed-loop state is exactly
$\Pi_nX^*$ for the fixed datum $\Pi_n\xi$, because the diagonal
coefficients commute with $\Pi_n$. The homogeneous mild-state estimate
bounds its omitted tail by $C_T\mathbb E\|(I-\Pi_n)\xi\|^2$.
The bounded feedbacks give convergence of the inputs in $L^2_{\mathbb F}$.
Continuity of the quadratic costs in these spaces proves cost convergence.
This is a fixed-datum convergence result, independent of the changing
initial states used in the numerical table below.
\end{pf}

\subsection{Spectral Galerkin System}

Let $e_k(x)=\sqrt2\sin(k\pi x)$, so that $Ae_k=-k^2\pi^2e_k$, and let
$H_n=\operatorname{span}\{e_1,\ldots,e_n\}$. Denote by $\Pi_n:H\to H_n$ the
orthogonal projection and set
\[
\xi_n=\Pi_n\xi,\qquad u_n=\Pi_nu,\qquad v_n=\Pi_nv.
\]
The Galerkin state $X_n$ is defined as the unique solution of
\[
\begin{cases}
\begin{aligned}
dX_n(s)&=\big[(A_n+A_{1,n})X_n(s)+B_{1,n}u_n(s)\\
&+C_{1,n}v_n(s)\big]\,ds+\big[A_{2,n}X_n(s)+B_{2,n}u_n(s)\\
&+C_{2,n}v_n(s)\big]\,dW(s),
\end{aligned}
\\
z_n(s)=\begin{pmatrix}M_nX_n(s)\\N_nu_n(s)\end{pmatrix},\qquad
X_n(t)=\xi_n,
\end{cases}
\]
where
\[
A_n=\mathrm{diag}(-\pi^2,-4\pi^2,\ldots,-n^2\pi^2),
\]
and $F_n=\Pi_nF|_{H_n}$ for every bounded coefficient $F$. In the basis
$\{e_1,\ldots,e_n\}$,
\[
\begin{aligned}
&A_{1,n}=0.1I_n,\quad A_{2,n}=0.2I_n,\\
&B_{1,n}=I_n,\quad B_{2,n}=0.5I_n,\\
&C_{1,n}=C_{2,n}=0.3I_n,\quad M_n=N_n=I_n.
\end{aligned}
\]

Given the selected coefficients, $H_n$ is invariant under $A$ and all bounded operators, making the exact solution projection coincide with the Galerkin approximation. Computationally, $X_n$ is governed directly by the preceding finite-dimensional system rather than the geometric projection $X_n=\Pi_nX$. This projected model forms a finite-dimensional stochastic system with state-, control-, and disturbance-dependent diffusion. Its finite-horizon mixed $H_2/H_\infty$ formulation and coupled Riccati structure rigorously align with established finite-dimensional Nash game frameworks \cite{chen2004stochastic,zhang2006xuv,zhang2017stochastic_nash}. This section explicitly bridges these finite-dimensional theories with the infinite-dimensional spectral projection.
\subsection{Projected Coupled Riccati Equations}
\label{subsec:galerkin_riccati}

Let
\[
\begin{aligned}
&P_{1,n},P_{2,n}\in C^1([t,T];\mathbb S^n),\\
&K_{U,n},K_{V,n}\in L^\infty(t,T;\mathbb R^{n\times n}).
\end{aligned}
\]
Define
\[
\begin{aligned}
\mathcal A_{U,n}&=A_n+A_{1,n}+B_{1,n}K_{U,n},\\
\mathcal C_{U,n}&=A_{2,n}+B_{2,n}K_{U,n},\\
\mathcal A_{V,n}&=A_n+A_{1,n}+C_{1,n}K_{V,n},\\
\mathcal C_{V,n}&=A_{2,n}+C_{2,n}K_{V,n},\\
\Lambda_{V,n}&=\gamma^2I_n+C_{2,n}^{\top}P_{1,n}C_{2,n},\\
\Upsilon_{V,n}&=P_{1,n}C_{1,n}
+\mathcal C_{U,n}^{\top}P_{1,n}C_{2,n},\\
\Lambda_{U,n}&=I_n+B_{2,n}^{\top}P_{2,n}B_{2,n},\\
\Upsilon_{U,n}&=P_{2,n}B_{1,n}
+\mathcal C_{V,n}^{\top}P_{2,n}B_{2,n}.
\end{aligned}
\]
The matrix counterpart of \eqref{eq:riccati_v} is
\begin{equation}\label{eq:galerkin_riccati_v}
\begin{aligned}
&0={}\dot P_{1,n}+\mathcal A_{U,n}^{\top}P_{1,n}+P_{1,n}\mathcal A_{U,n}+\mathcal C_{U,n}^{\top}P_{1,n}\mathcal C_{U,n}\\
&-M_n^{\top}M_n-K_{U,n}^{\top}K_{U,n}-\Upsilon_{V,n}\Lambda_{V,n}^{-1}
\Upsilon_{V,n}^{\top},\\
& P_{1,n}(T)=0.
\end{aligned}
\end{equation}
The matrix counterpart of \eqref{eq:riccati_u} is
\begin{equation}\label{eq:galerkin_riccati_u}
\begin{aligned}
&0={}\dot P_{2,n}+\mathcal A_{V,n}^{\top}P_{2,n}+P_{2,n}\mathcal A_{V,n}+\mathcal C_{V,n}^{\top}P_{2,n}\mathcal C_{V,n}\\
&+M_n^{\top}M_n-\Upsilon_{U,n}\Lambda_{U,n}^{-1}
\Upsilon_{U,n}^{\top},\\
& P_{2,n}(T)=0.
\end{aligned}
\end{equation}
These equations are coupled with the feedback identities
\[
\begin{aligned}
K_{V,n}&=-\Lambda_{V,n}^{-1}\Upsilon_{V,n}^{\top},\\
K_{U,n}&=-\Lambda_{U,n}^{-1}\Upsilon_{U,n}^{\top}.
\end{aligned}
\]
Both projected Riccati equations are terminal-value problems with
$P_{1,n}(T)=P_{2,n}(T)=0$, so any numerical solution must integrate them backward
from $T$; no forward-time integration is used.
If the coupled system consisting of
\eqref{eq:galerkin_riccati_v}--\eqref{eq:galerkin_riccati_u} and the two feedback
identities admits a solution in the mild class stated above, and if the
following coercivity bounds hold uniformly in time,
\[
\begin{aligned}
\Lambda_{V,n}&\ge\delta_{1,n}I_n,\\
\Lambda_{U,n}&\ge\delta_{2,n}I_n,
\qquad\delta_{1,n},\delta_{2,n}>0.
\end{aligned}
\]
then the projected verification theorem yields the finite-dimensional feedback laws
\[
u_n^*(s)=K_{U,n}(s)X_n^*(s),\qquad
v_n^*(s)=K_{V,n}(s)X_n^*(s).
\]

\subsection{Numerical Verification}
\label{subsec:numerics}

We report reproducible spectral-Galerkin computations for the coefficient
choice of Section~\ref{subsec:heat_model} with $t=0$, $T=1$, and
$U=V=Z_1=Z_2=H$. Because all coefficients are scalar multiples of the
identity, the projected modes decouple exactly: every mode $k$ satisfies one scalar
$H_\infty$ Riccati equation, one scalar $H_2$ Riccati equation, and the two
scalar feedback identities $K_{V,k}=-\Lambda_{V,k}^{-1}\Upsilon_{V,k}$ and
$K_{U,k}=-\Lambda_{U,k}^{-1}\Upsilon_{U,k}$. We use the forward variable
$\tau=T-s$ for the terminal-value equations. The coupled Riccati equations are
integrated by the adaptive implicit Radau method implemented in SciPy, and the $2\times2$
linear system formed by the feedback identities is solved at every right-hand-side
evaluation. Fixed-feedback bounded-real equations are integrated with the same tolerances,
and both the induced gain and the endpoint of the coercive branch are located by
bisection. The reproduction script \texttt{reproduce\_numerics.py} records
the solver version, tolerances, bisection brackets, and sampled margins in
\texttt{numerical\_results.json}. The default relative and absolute ODE
tolerances are $2\times10^{-9}$ and $2\times10^{-11}$; the gain bisection
stops at bracket width $2\times10^{-8}$. A solve is rejected if the
effective disturbance weight reaches $10^{-10}$ or if integration fails.
This cutoff is a numerical branch criterion, not an interval-arithmetic
certificate. Dense output is sampled at $4001$ equally spaced time points
for margin and stationarity diagnostics. All reported nominal
$J_2$ values use $v\equiv0$; by contrast, every induced gain is the supremum over
nonzero admissible disturbances, estimated here through the BRL equivalence.
No deterministic-input restriction is used and no independent stochastic
input--output norm computation is claimed.
Tightening the relative and absolute tolerances by a factor of ten leaves
the displayed baseline gains, the branch endpoint, and the six-decimal
nominal costs unchanged. This sensitivity check is recorded in
\texttt{numerical\_results\_refined.json}; it is not a certified error bound.
Use the script option \texttt{--all-modal-gains} to reproduce the gain
columns for every retained mode in Table~\ref{tab:galerkin}.

The script integrates the modal system
\eqref{eq:heat_modal_odes}, using $g=\gamma$ in parameter studies,
with $\tau=T-s$ and zero initial Riccati
weights. At a fixed control feedback, the BRL equation is its first
equation with $k_u$ prescribed and $g$ replaced by the trial attenuation
level. The nominal cost is $q(T)$,
where $q(0)=0$ and
\[
\frac{dq}{d\tau}=[2(a_k+k_u)+(0.2+0.5k_u)^2]q+1+k_u^2.
\]
These formulas distinguish the nominal value from the coupled value
$p_2(T)$ and make the role of the full output weight $1+k_u^2$ explicit.

\subsubsection{The \texorpdfstring{$\gamma$}{gamma}-threshold and the bounded real lemma}
For the single disturbance channel ($u=0$), the computed modal gains attain
their maximum at $k=1$ for the coefficient choice above.
Fig.~\ref{fig:brl} shows the behavior of the strongly regular mild solution
of \eqref{eq:hinf_riccati} as a function of $\gamma$. The norm $|P(0)|$ and
the coercivity margin $\min_s\Lambda^\gamma(P(s))$ remain finite and positive,
respectively, at the sampled levels $\gamma>\gamma^*\approx0.088407$. As
$\gamma\downarrow\gamma^*$, $|P(0)|$ increases to the finite boundary value
$(\gamma^*)^2/0.09\approx0.08685$, whereas the coercivity margin vanishes.
Thus it is loss of strong regularity, not blow-up of $P(0)$, that marks the
threshold. Bisection of the strongly regular Riccati branch gives the threshold estimate $0.088407$. By Theorem~\ref{thm:sbrl}, this threshold estimates $\|\mathcal L\|$; it is a Riccati-based gain computation, not an independent numerical verification of the theorem.

\begin{figure*}[!t]
\centering
\includegraphics[width=\textwidth]{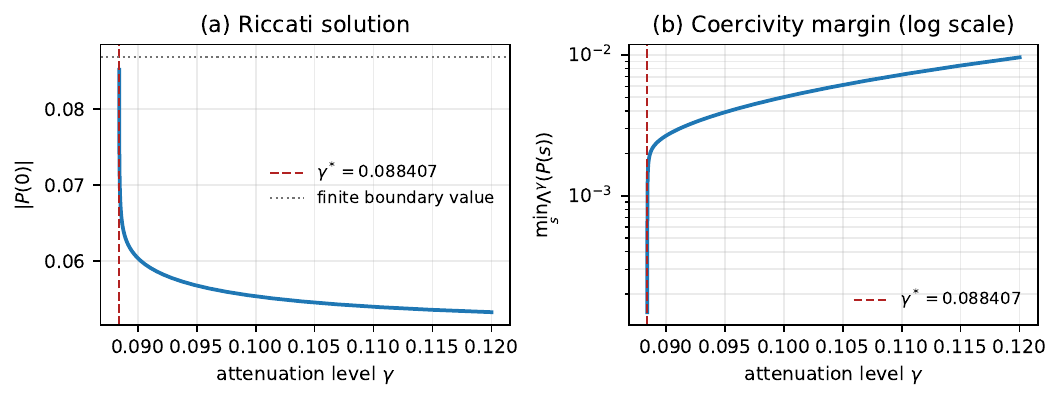}
\caption{The single-channel bounded real lemma at mode $k=1$. (a) Norm $|P(0)|$
of the strongly regular mild solution of \eqref{eq:hinf_riccati} as a function
of $\gamma$: it approaches a finite boundary value as
$\gamma\downarrow\gamma^*$. (b) Coercivity margin
$\min_s\Lambda^\gamma(P(s))$, shown on a logarithmic vertical scale: it
vanishes as $\gamma\downarrow\gamma^*$. The dashed vertical line marks
$\gamma^*=\|\mathcal L\|\approx0.088407$.}
\label{fig:brl}
\end{figure*}

\subsubsection{Solvability of the cross-coupled Riccati equations}
Fig.~\ref{fig:coupled} reports the analogous experiment for the cross-coupled
Riccati equations \eqref{eq:galerkin_riccati_v}--\eqref{eq:galerkin_riccati_u}.
These equations can be continued on a coercive numerical branch for
$\gamma>\gamma^*_c\approx0.088532$, with strictly positive margins
$\min_s\Lambda_V(P_1(s))>0$ and $\min_s\Lambda_U(P_2(s))>0$. Below the displayed
threshold our continuation loses positivity of $\Lambda_V$; this is evidence
for the endpoint of the computed strongly regular branch, not a proof that no
other Riccati pair exists. Throughout the sampled branch the achieved
closed-loop gain is strictly below the prescribed level and retains a positive
gap from $\gamma$ near $\gamma^*_c$; the calculation therefore does not support
an asymptotic equality between the two quantities. At $\gamma=0.09$,
$\min_s\Lambda_V=0.002791$, $\min_s\Lambda_U=1$, and
$K_U(0)=-0.1991$, whereas $K_V(0)=6.9762$. The relatively large disturbance
gain also indicates increasing sensitivity near the branch endpoint. These
computations are consistent with a parameter range satisfying the
finite-dimensional counterpart of Assumption~\ref{assum:cross_riccati}.
The projected theorem applies if the exact Riccati solution has the
required bounds; sampled positivity alone does not rigorously certify
those bounds. The numerical branch endpoint is not an analytic necessity threshold. At $\gamma=0.09$, the infinite-dimensional full-horizon assumption is independently established by Theorem~\ref{thm:heat_global}.

\begin{figure*}[!t]
\centering
\includegraphics[width=\textwidth]{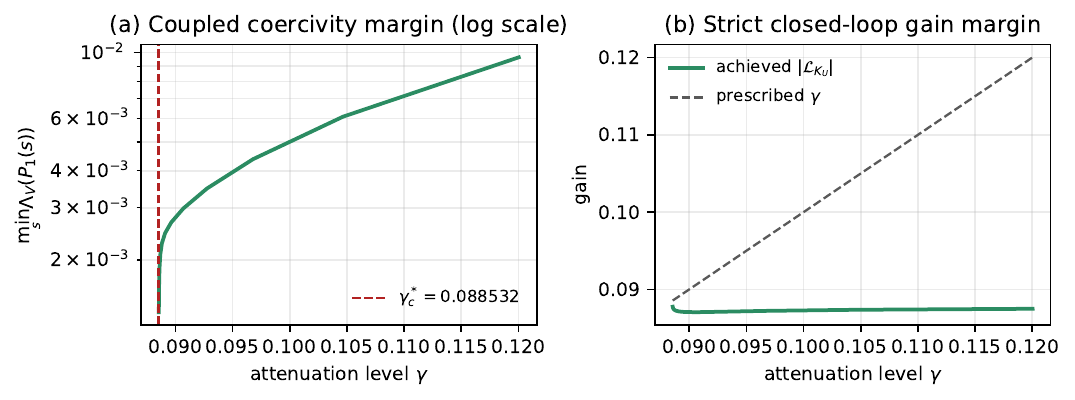}
\caption{The cross-coupled Riccati equations at mode $k=1$. (a) Minimum coercivity margin
$\min_s\Lambda_V(P_1(s))$ of the coupled $H_\infty$ equation as a function of
$\gamma$, shown on a logarithmic vertical scale; the dashed line marks the
endpoint $\gamma^*_c\approx0.088532$ of the computed coercive branch.
(b) Achieved closed-loop gain
$\|\mathcal L_{K_U}\|$ of the mixed design versus the prescribed level
$\gamma$: it remains strictly below $\gamma$ throughout the sampled branch.}
\label{fig:coupled}
\end{figure*}

\subsubsection{Baseline comparison}
Table~\ref{tab:baseline} compares, at $\gamma=0.09$, the four design cases on
the first mode, which is the worst among the computed modes. For each row, the induced gain is computed over all
nonzero disturbances, whereas the nominal cost uses $X(0)=e_1$ and $v\equiv0$.
The LQ regulator attains the smallest nominal cost, but the attenuation level is
not imposed during its synthesis; its strict-gain property is verified only
a posteriori by the fixed-feedback bounded-real Riccati equation. A bounded
scalar search over constant control gains in $[-0.6,0.2]$ gives a gain near
$0.087061$ and nominal cost near $0.05187$. The search is numerical and does
not certify a global optimum over all feedback strategies. The mixed
controller gives the reported gain $0.087047$ and is generated by the
cross-coupled Riccati equations. This small difference is not evidence of
global gain optimality: the constant-gain and time-varying designs use
different search classes, and neither optimization over all admissible
feedbacks nor a certified error bound is supplied.
Its displayed nominal cost $0.052169$ is evaluated with $v=0$ and must not be
confused with the Nash value against $v=K_VX$, which is
$P_2(0)=0.094943$ for $X(0)=e_1$. The coupled Riccati identity verifies the
corresponding best-response inequalities, while the reported nominal comparison
evaluates all four controllers under a common disturbance-free condition.

We also integrate the fixed-feedback scalar moment-cost equations directly,
rather than merely reading off the Riccati values. For $X(0)=e_1$,
the equilibrium costs are $J_1=-0.058985$ and $J_2=0.094943$, agreeing
with $P_1(0)$ and $P_2(0)$. Holding $K_U$ fixed and replacing $K_V$ by
$0$, $0.5K_V$ or $1.5K_V$ gives $J_1=-0.052169$, $-0.056931$ or
$-0.049280$, respectively. Holding $K_V$ fixed and using zero control
or the nominal LQ controller gives $J_2=0.098935$ or $0.096968$.
These selected deviations are consistent with the Nash inequalities;
they do not numerically exhaust all predictable deviations. The full
inequalities follow from the verification theorem under its hypotheses.

\begin{table*}[!t]
\centering
\caption{Closed-loop disturbance-attenuation gain and nominal quadratic cost
for four designs (mode $k=1$, $\gamma=0.09$). All nominal costs use
$X(0)=e_1$ and $v=0$. ``Enforced'' indicates whether the attenuation constraint
enters the synthesis; ``BRL check'' records a numerical bounded-real
Riccati check with a positive sampled coercivity margin. Case D additionally
has the analytic certificate of Theorem~\ref{thm:heat_global}; the numerical
checks alone are not rigorous certificates.}
\label{tab:baseline}
\begin{tabular}{clcccc}
\toprule
Case & Controller for $u$ & estimated $\Vert\mathcal L_{K_U}\Vert$ & nominal $J_2$ & Enforced? & BRL check \\
\midrule
A & open loop, $K_U=0$ & 0.0884 & 0.0513 & no & yes (uncoupled) \\
B & LQ feedback $K_U^{\rm LQ}$ & 0.0877 & 0.0511 & no & yes (a posteriori) \\
C & constant-gain $H_\infty$ feedback & 0.0871 & 0.0519 & yes & yes (uncoupled) \\
D & mixed feedback $K_U$ & 0.0870 & 0.0522 & yes & yes (analytic) \\
\bottomrule
\end{tabular}
\end{table*}

\subsubsection{Projection-order study}
Table~\ref{tab:galerkin} reports the projected quantities for increasing $n$ at
$\gamma=0.09$. The scalar-multiple coefficients diagonalize the projected
problem, and the computed open-loop gain, mixed gain, and minimum coercivity
constant attain their worst values on the first mode for every projection order
reported in the table. Consequently, these quantities stabilize already at
$n=1$. This is a numerical property of the present coefficient choice; modal
decoupling alone is not being invoked as a general monotonicity theorem. For the
last column we deliberately choose the changing unit
initial state
\[
\xi_n=\frac{1}{\sqrt n}\sum_{k=1}^n e_k
\]
and evaluate the mixed controller with $v=0$. The resulting quantity is the
average of the first $n$ modal nominal costs and decreases because increasingly
damped modes receive part of the fixed unit energy. Since the initial state
changes with $n$, this column is a projection-order diagnostic rather than a
Galerkin convergence test for a fixed datum. The invariance of the worst mode
nevertheless confirms the mode-by-mode structure of this diagonal example.
For each finite projection, the induced gain is the maximum of the modal
gains: orthogonality makes both input and output energies sums of modal
energies, giving the upper bound, and an input in a maximizing mode gives
the lower bound. This reasoning does not require independent modal noise;
all modes here share the same Brownian motion.

\begin{table*}[!t]
\centering
\caption{Projection-order study at $\gamma=0.09$: worst-mode open-loop gain
$\Vert\mathcal L\Vert_n$, worst-mode closed-loop gain of the mixed design,
minimum coercivity constant of $\Lambda_V(P_1)$, and the nominal cost of the
mixed controller under $v=0$ for
$\xi_n=n^{-1/2}\sum_{k=1}^n e_k$.}
\label{tab:galerkin}
\begin{tabular}{lcccc}
\toprule
$n$ & $\Vert\mathcal L\Vert_n$ & $\Vert\mathcal L_{K_U}\Vert_n$ & $\min_{k\le n}\min_s\Lambda_V(P_{1,k})$ & nominal $J_2(\xi_n)$ \\
\midrule
1 & 0.0884 & 0.0870 & 0.0028 & 0.052169 \\
2 & 0.0884 & 0.0870 & 0.0028 & 0.032435 \\
5 & 0.0884 & 0.0870 & 0.0028 & 0.015141 \\
10 & 0.0884 & 0.0870 & 0.0028 & 0.008007 \\
20 & 0.0884 & 0.0870 & 0.0028 & 0.004121 \\
50 & 0.0884 & 0.0870 & 0.0028 & 0.001678 \\
\bottomrule
\end{tabular}
\end{table*}

\subsection{Scope of the Illustration}

Theorem~\ref{thm:heat_global} establishes a full-horizon diagonal Nash
design and operator-norm spectral convergence for the specified coefficients
at $\gamma=0.09$. This analytic result covers all modes, including those
beyond the numerical truncation, and supplies a uniform disturbance-weight
margin of $0.00252$. The changing-datum cost column in
Table~\ref{tab:galerkin} remains a projection-order diagnostic; it is not
the fixed-datum convergence result proved in that theorem.

The numerical studies of the branch near $\gamma_c^*$ concern other
attenuation levels. Neither the invariant rectangle used at $0.09$ nor
the local contraction identifies a global solvability threshold over
$\gamma$, rules out other branches, or proves necessity of the numerical
endpoint. For noncommuting coefficients, rigorous approximation still
requires uniform bounds and control of nonlinear operator products;
the diagonal convergence result is not claimed in that generality.

\section{Concluding Remarks}

The finite-horizon mixed $H_2/H_\infty$ design studied here is a
nonzero-sum feedback Nash problem. Its $H_2$ channel is an LQ energy.
The single-channel bounded-real equivalence specializes the stochastic
LQ uniform-convexity characterization. A simultaneous bounded-generator
approximation establishes the quadratic identities in the strongly
continuous mild operator class, and a coupled Riccati pair with uniformly
positive weights verifies both Nash inequalities and full-output strict
attenuation.

For every fixed $\gamma>0$, invertibility of the feedback block near the
zero terminal weights and a contraction of the coupled mild integral map
give a locally unique solution on a sufficiently short terminal interval.
This is an existence result for the infinite-dimensional coupled problem,
but does not guarantee continuation to an arbitrary prescribed horizon.
For the heat-equation example at $\gamma=0.09$, a uniform invariant rectangle gives full-horizon coupled solvability, bounded infinite-dimensional feedbacks and operator-norm spectral convergence. The computations at other attenuation levels remain branch diagnostics. Neither the theory nor the comparison table asserts global constrained $H_2$ optimality.

Further work should obtain continuation criteria and convergent Galerkin feedback approximations for general noncommuting coefficients.
Multidimensional noise and unbounded input or observation operators
require additional analysis.

\end{document}